\documentclass[11pt]{article}%
\usepackage[letterpaper, portrait, margin=1in, headheight=0pt]{geometry}
\usepackage{mathtools}
\usepackage{enumerate}
\usepackage{amsmath,amsthm,amssymb,amsfonts}
\usepackage{enumitem}
\usepackage{mathrsfs}
\usepackage{tikz-cd}
\usepackage{multicol}
\usepackage[title]{appendix}
\usetikzlibrary{quotes,angles}

\DeclareMathOperator{\Hom}{Hom}

\DeclareMathOperator{\ad}{ad}

\DeclareMathOperator{\ima}{Im}
\DeclareMathOperator{\Der}{Der}
\DeclareMathOperator{\Fact}{Fact}
\DeclareMathOperator{\Ext}{Ext}

\DeclareMathOperator{\Cext}{Cext}

\newcommand{\F}{\mathbb{F}}

\newcommand{\inv}{^{-1}}
\newcommand{\vp}{\varphi}
\newcommand{\LL}{L}

\newcommand{\alg}{\mathscr{P}}
\newcommand{\fa}{A}
\newcommand{\fb}{B}

\newcommand{\cZ}{\mathcal{Z}}
\newcommand{\cB}{\mathcal{B}}
\newcommand{\cH}{\mathcal{H}}
\newcommand{\cC}{\mathcal{C}}
\newcommand{\T}{\mathscr{T}}
\newcommand{\vv}{_{\vdash}}
\newcommand{\dd}{_{\dashv}}

\newtheorem{thm}{Theorem}[section]

\newtheorem{cor}[thm]{Corollary}

\theoremstyle{definition}
\newtheorem{defn}{Definition}

\theoremstyle{remark}

\title{Nonabelian Extensions and Factor Systems \\ for the Algebras of Loday}
\author{Erik Mainellis}
\date{}

\begin{document}

\maketitle

\begin{abstract}
    Factor systems are a tool for working on the extension problem of algebraic structures such as groups, Lie algebras, and associative algebras. Their applications are numerous and well-known in these common settings. We construct $\alg$ algebra analogues to a series of results from W. R. Scott’s \textit{Group Theory}, which gives an explicit theory of factor systems for the group case. Here $\alg$ ranges over Leibniz, Zinbiel, diassociative, and dendriform algebras, which we dub ``the algebras of Loday,'' as well as over Lie, associative, and commutative algebras. Fixing a pair of $\alg$ algebras, we develop a correspondence between factor systems and extensions. This correspondence is strengthened by the fact that equivalence classes of factor systems correspond to those of extensions. Under this correspondence, central extensions give rise to 2-cocycles while split extensions give rise to (nonabelian) 2-coboundaries.
\end{abstract}

\tableofcontents

\section{Introduction}
This paper is something of a tribute to Loday, who first generated interest in Leibniz algebras as a generalization of Lie algebras. Loday also defined dual Leibniz algebras \cite{loday cup product}, which later took the name Zinbiel algebras based on Loday's pen name G. W. Zinbiel, or ``Leibniz" backwards. He wrote under this name in ``Encyclopedia of Types of Algebras 2010'' \cite{zinbiel}, a paper that lists various algebras and some of their properties from an ``operadic point of view." Loday also introduced the notions of diassociative algebras (or associative dialgebras) and dendriform algebras in his paper ``Dialgebras'' \cite{loday dialgebras}. The algebras of Loday fit nicely into the following commutative diagram of functors between the categories Zinb, Dend, Com, As, Dias, Lie, and Leib of Zinbiel, Dendriform, Commutative, Associative, Diassociative, Lie, and Leibniz algebras respectively. Said diagram depicts the symmetry of their corresponding operads under the Koszul duality, which falls across a vertical line through the category As. We take this diagram from Loday \cite{loday dialgebras}, although it has also appeared in a number of other papers concerning operads.

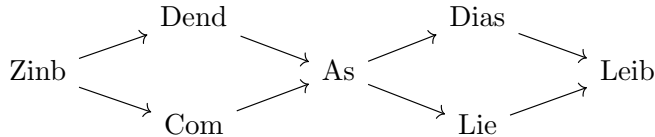
\[\begin{tikzcd}[row sep=tiny]
& \text{Dend} \arrow[dr] & & \text{Dias} \arrow[dr] \\ \text{Zinb} \arrow[ur] \arrow[dr] & & \text{As} \arrow[dr] \arrow[ur] & & \text{Leib} \\
& \text{Com} \arrow[ur] & & \text{Lie}\arrow[ur]
\end{tikzcd}\]

The objective of this paper is to establish an explicit theory of factor systems, also known as nonabelian 2-cocycles, for the algebras of Loday. Factor systems are a classic tool for working on the extension problem of algebraic structures. These concepts originated in the group setting and can be traced back to Schreier's 1926 paper \cite{schreier}.\footnote{See also the introduction of \cite{brown} for some discussion of factor systems and the extension problem.} In particular, consider a pair of $\alg$ algebras $A$ and $B$ (throughout, $\alg$ will range over the seven algebras above). The extension problem concerns the classification of all algebras $L$ such that $0\xrightarrow{} A\xrightarrow{} L\xrightarrow{} B\xrightarrow{} 0$ is an extension of $A$ by $B$. Factor systems yield an explicit multiplication on $L$ and provide concrete tools for dealing with the relations among $L$, $A$, and $B$. These tools are commonly assumed in the well-known settings, although are rarely justified in detail. Nevertheless, the implications of factor systems are considerable, a fact that has been heavily demonstrated in the settings of associative and Lie algebras (and, of course, groups).

The majority of applications for factor systems have involved the special case of central extensions. The corresponding central factor systems, also known as 2-cocycles, are, for example, used in \cite{willem} to classify nilpotent associative algebras up to dimension 4. In \cite{million}, the authors develop a method for classifying nilpotent Lie algebras that relies on central factor systems. Central factor systems give rise to the second cohomology group $\cH^2(B,A)$ of an algebra $B$ with coefficients in a $B$-module $A$, yielding a characterization of central extensions by $\cH^2(B,A)$. This, in turn, gives rise to nice exact sequences (see \cite{karel} for example). The identities for factor systems of Lie algebras appear frequently in Lie theory. One such instance can be found in a section of Jacobson's \textit{Lie Algebras} \cite{jacobson} which concerns the theorems of Levi and Malcev-Harish-Chandra. On the other hand, a formulation of nonabelian factor systems for associative algebras is given in \cite{gouray}, where the group $\cH_{nab}^2(B,A)$ of nonabelian factor systems is shown to be in 1-1 correspondence with Maurer-Cartan elements of a certain differential graded Lie algebra. Such a feat can also be performed in the Leibniz setting (see \cite{liu}).

With these applications in mind, the intent of this paper is to serve as a starting point from which to advance extension theory in the settings of less-developed algebras. We continue to find results, some of which have been listed above, that rely on the assumptions of factor systems but have only been established for certain algebras. Furthermore, it is important to construct the more general nonabelian factor systems, which allows for work on noncentral extensions. Indeed, one of our current efforts involves extensions of nilpotent algebras, for which there are known results in the Lie setting that rely on factor systems. We plan to develop said results for all seven algebras in a future paper.

The final motivation for this paper is to appreciate the structure of factor systems themselves. In particular, consider the study of Leibniz algebras. We observe that much of Leibniz theory involves the generalization of Lie-theoretic results, which often carry over to the Leibniz setting with only slight adjustments. For example, there are entire proofs in Lie theory which hold for Leibniz algebras merely by replacing the bracket Lie algebra $[A,B]$ with the sum Leibniz algebra $AB+BA$. Factor systems, however, are a case in which generalizing gives rise to considerably more complicated structures. One powerful aspect of our functorial diagram is that each northeast movement along its arrows corresponds to a generalization of algebra type. This fact can be reasoned by seeing each algebra as a special case of its northeast-adjacent category. Besides the usual Lie to Leibniz comparison, any Zinbiel algebra can be seen as a dendriform algebra with multiplications satisfying $x<y = y>x$. Next, any commutative algebra is simply an associative algebra in which $xy=yx$. Finally, any associative algebra can be seen as a diassociative algebra in which $x\dashv y = x\vdash y$. Thus, any result that holds for the Leibniz, diassociative, and dendriform cases must necessarily hold for \textit{all seven algebras}. The definitions of all seven factor systems are provided in this paper for the sake of structural comparison.

To develop a theory of factor systems, we take our methodology from a chapter in W. R. Scott's \textit{Group Theory} \cite{scott} which develops and investigates the correspondence between factor systems and extensions of groups. This research began as an effort to develop a Leibniz analogue of said chapter in order to work on extensions of nilpotent algebras. We later discovered that the Leibniz case had already appeared in a 2018 paper \cite{liu} where a similar list of identities was obtained for nonabelian 2-cocycles. We then turned to the Zinbiel, diassociative, and dendriform cases, which we have not been able to find in written form. We have decided, however, to review the Leibniz case in explicit detail. The purpose is to provide a systematic approach to this theory as well as a self-contained paper with consistent notation. For the more complicated and obscure algebras, the results are thereby easier to follow.

This paper is structured as follows. For preliminaries, we define each type of algebra under consideration and provide a discussion of general extension theory for arbitrary algebras. The main work begins with the Leibniz analogue of the results in \cite{scott}. We next derive the diassociative analogue. These results can also be constructed for dendriform algebras via the same process as the diassociative case, replacing $\dashv$ with $<$ and $\vdash$ with $>$, although the identities which appear are uniquely determined by the different algebra structures. Thus do our results hold for all seven algebras via the above discussions. The definitions of the remaining factor systems are provided in the same sections as their generalizations. We conclude this paper with a brief discussion of the usual Leibniz cohomology and show how $\cH^2(B,A)$ characterizes extensions. The paper ends with a list of 2-cocycles and their defining identities. Said identities have appeared in \cite{rak}, but we collect them here for the sake of completeness.

\section{Preliminaries}
Let $\F$ be a field. All algebras will be $\F$-vector spaces equipped with bilinear multiplications which satisfy certain identities. We first recall that a \textit{Lie algebra} $L$ has multiplication which is \textit{alternating} and satisfies the \textit{Jacobi identity}. Respectively, this means that $xx = 0$ and $(xy)z + (yz)x + (zx)y = 0$ for all $x,y,z\in L$. A \textit{Leibniz algebra} $L$ is another type of nonassociative algebra whose multiplication satisfies the \textit{Leibniz identity} $x(yz) = (xy)z + y(xz)$ for all $x,y,z\in L$. It is well-known that any multiplication which is alternating is also skew-symmetric (meaning $xy=-yx$). Under skew-symmetry, the Leibniz identity can be rearranged to form the Jacobi identity, and thus Leibniz algebras are famously seen as the non-anticommutative generalization of Lie algebras.

\begin{defn}
A \textit{Zinbiel algebra} $Z$ is a nonassociative algebra with multiplication satisfying what we will call the \textit{Zinbiel identity} $(xy)z = x(yz) + x(zy)$ for all $x,y,z\in Z$.
\end{defn}

\begin{defn}
A \textit{diassociative algebra} $D$ is a vector space equipped with two associative bilinear products $\dashv$ and $\vdash$ which satisfy the following identities for all $x,y,z\in D$:
\begin{enumerate}
    \item[D1.] $x\dashv (y\dashv z) = x\dashv (y\vdash z)$,
    \item[D2.] $(x\vdash y)\dashv z = x\vdash (y\dashv z)$,
    \item[D3.] $(x\dashv y)\vdash z = (x\vdash y)\vdash z$.
\end{enumerate}
\end{defn}

\begin{defn}
A \textit{dendriform algebra} $E$ is a vector space equipped with two bilinear products $<$ and $>$ which satisfy the following identities for all $x,y,z\in E$:
\begin{enumerate}
    \item[E1.] $(x<y)<z = x<(y<z) + x<(y>z)$,
    \item[E2.] $(x>y)<z = x>(y<z)$,
    \item[E3.] $(x<y)>z + (x>y)>z = x>(y>z)$.
\end{enumerate}
\end{defn}

We now review extensions. Let $\alg$ denote a type of algebra (e.g. Lie, Leibniz, diassociative) and let $A$ and $B$ be arbitrary $\alg$ algebras. An \textit{extension of $A$ by $B$} is a short exact sequence of the form $0\xrightarrow{} A\xrightarrow{\sigma} \LL\xrightarrow{\pi} B\xrightarrow{} 0$ where $\LL$ is a $\alg$ algebra and $\sigma$ and $\pi$ are \textit{homomorphisms}, i.e. linear maps that preserve the $\alg$ structure. A \textit{section} of the extension is a linear map $T:B\xrightarrow{} L$ such that $\pi T = I_B$. Two extensions $0\xrightarrow{} A\xrightarrow{\sigma_1} \LL_1\xrightarrow{\pi_1} B\xrightarrow{} 0$ and $0\xrightarrow{} A\xrightarrow{\sigma_2} \LL_2\xrightarrow{\pi_2} B\xrightarrow{} 0$ of $A$ by $B$ are called \textit{equivalent} if there exists an isomorphism $\tau:\LL_1\xrightarrow{} \LL_2$ such that the diagram \[\begin{tikzcd}%[row sep = small]
0\arrow[r] &A\arrow["\sigma_1",r] \arrow["\text{id}",d] & \LL_1 \arrow[d,"\tau"] \arrow[r,"\pi_1"] & B\arrow[r] \arrow["\text{id}",d]&0\\
0\arrow[r] &A\arrow["\sigma_2",r,swap] & \LL_2 \arrow["\pi_2",r,swap] &B\arrow[r] &0
\end{tikzcd}\] commutes, i.e. if $\tau\sigma_1 = \sigma_2$ and $\pi_2\tau = \pi_1$. The isomorphism $\tau$ is called an \textit{equivalence}. An extension $0\xrightarrow{} A\xrightarrow{\sigma} \LL \xrightarrow{\pi} B\xrightarrow{}0$ of $A$ by $B$ is said to \textit{split} if there exists a homomorphism $T:B\xrightarrow{} \LL$ which is also a section. An extension $0\xrightarrow{} A\xrightarrow{\sigma} \LL \xrightarrow{\pi} B\xrightarrow{}0$ is called \textit{central} if $\sigma(A)$ is contained in the center $Z(L)$ of $L$, and \textit{abelian} if $\LL$ is abelian.

It is readily verified that equivalence of extensions is an equivalence relation. Moreover, if an extension splits, then so does every equivalent extension. Indeed, let $0\xrightarrow{} A\xrightarrow{\sigma_1} \LL_1\xrightarrow{\pi_1} B\xrightarrow{} 0$ be a split extension which is equivalent to another extension $0\xrightarrow{} A\xrightarrow{\sigma_2} \LL_2\xrightarrow{\pi_2} B\xrightarrow{} 0$ via an equivalence $\tau$. Then there is a homomorphism $T_1:B\xrightarrow{} \LL_1$ such that $\pi_1 T_1 = I_B$, which implies that $T_2=\tau T_1$ is a homomorphism from $B$ into $\LL_2$ satisfying $\pi_2T_2 = \pi_2\tau T_1 = \pi_1 T_1 = I_B$. Finally, since an equivalence is an isomorphism, it is straightforward to verify that extensions which are equivalent to abelian extensions are abelian, and extensions which are equivalent to central extensions are central.

\section{Factor Systems of Leibniz Algebras}
As with the group case, a factor system of algebras is a tuple of maps together with a set of identities that said maps satisfy. In particular, a Leibniz factor system involves three maps and seven defining identities while a Lie factor system involves only two maps and three identities. The Lie identities are, of course, well-known.
%The purpose of this section is to provide an explicit model for understanding the theories of other algebras in Scott's framework. We also desire a self-contained and complete paper.
The definitions of factor systems for both Lie and Leibniz algebras are stated here for the sake of comparison. Recall that $\ad^l$ and $\ad^r$ denote the left and right multiplication operators respectively; $\ad^l$ is simply called $\ad$ in the Lie case since $\ad^r = -\ad$.

\begin{defn} Let $A$ and $B$ be Leibniz algebras. A \textit{factor system} of $A$ by $B$ is a tuple of maps $(\vp,\vp',f)$ where \begin{itemize}
    \item[] $\vp:B\xrightarrow{} \Der(A)$ is linear,
    \item[] $\vp':B\xrightarrow{} \mathscr{L}(A)$ is linear,
    \item[] $f:B\times B\xrightarrow{} A$ is bilinear
\end{itemize} such that \begin{enumerate}
    \item $m(\vp(i)n) = (\vp'(i)m)n + \vp(i)(mn)$
    \item $m(\vp'(i)n) = \vp'(i)(mn) + n(\vp'(i)m)$
    \item $\ad^r_{f(i,j)} + \vp'(ij) = \vp'(j)\vp'(i) + \vp(i)\vp'(j)$
    \item $\vp(i)(mn) = (\vp(i)m)n + m(\vp(i)n)$
    \item $\vp(i)\vp(j) = \vp(ij) + \vp(j)\vp(i) + \ad^l_{f(i,j)}$
    \item $\vp(i)\vp'(j) = \vp'(j)\vp(i) + \vp'(ij) + \ad^r_{f(i,j)}$
    \item $f(i,jk) + \vp(i)f(j,k) = f(ij,k) + \vp'(k)f(i,j) + f(j,ik) + \vp(j)f(i,k)$
\end{enumerate} are satisfied for all $m,n\in A$ and $i,j,k\in B$. Note that the fourth identity allows for $\vp:B\xrightarrow{} \Der(A)$.
\end{defn}

\begin{defn} Let $\fa$ and $\fb$ be Lie algebras with multiplications denoted by bracket and consider the natural Lie algebra structure on $\Der(\fa)$ under the commutator bracket. A \textit{factor system} of $\fa$ by $\fb$ is a pair $(\vp,f)$ of functions \begin{enumerate}
    \item[] $\vp:\fb\xrightarrow{} \Der(\fa)$ linear,
    \item[] $f:\fb\times \fb\xrightarrow{} \fa$ bilinear \end{enumerate} such that
\begin{enumerate}
    \item $f(i,j) = -f(j,i)$
    \item $\vp[i,j] = [\vp(i),\vp(j)]- \ad_{f(i,j)}$
    \item $\vp(k)f(i,j) + \vp(i)f(j,k) + \vp(j)f(k,i) = f([i,j],k) + f([j,k],i) + f([k,i],j)$
\end{enumerate} for all $i,j,k\in \fb$.
\end{defn}

\subsection{Belonging}
Our first goal is to construct a correspondence between factor systems and extensions. Consider an extension $0\xrightarrow{} A\xrightarrow{\sigma} \LL\xrightarrow{\pi} B\rightarrow{} 0$ of $A$ by $B$ and a section $T:B\xrightarrow{} \LL$. Consider also the linear maps $\rho:\LL\xrightarrow{} \Der(\sigma(A))$ and $\rho':\LL\xrightarrow{} \mathscr{L}(\sigma(A))$ defined by $\rho(x) = \ad^l_x|_{\sigma(A)}$ and $\rho'(x) = \ad^r_x|_{\sigma(A)}$ respectively for $x\in L$. Put simply, these maps denote the left and right multiplication operators that act on the image of $\sigma$ in $\LL$. We next use $\rho$ and $\rho'$ to define the maps $P:\LL\xrightarrow{} \Der(A)$ and $P':\LL\xrightarrow{} \mathscr{L}(A)$ by $P(x) = \sigma\inv \rho(x)\sigma$ and $P'(x) = \sigma\inv \rho'(x)\sigma$ respectively, formalizing a way for $\LL$ to act on $A$. To work explicitly with these maps, one computes $P(x)m = \sigma\inv \rho(x)\sigma(m) = \sigma\inv(x\sigma(m))$ and $P'(x)m = \sigma\inv \rho'(x)\sigma(m) = \sigma\inv(\sigma(m)x)$ for any $m\in A$ and $x\in \LL$. The maps $\vp$ and $\vp'$ of a factor system are ways for $B$ to act on $A$. It is thus natural to compose $P$ and $P'$ with $T$, as well as to define $f$ in terms of $T$, which leads to Definition \ref{belonging}. What follows are two converse results that form the framework for our correspondence.

\begin{defn}\label{belonging}
A factor system $(\vp,\vp',f)$ of $A$ by $B$ \textit{belongs} to the extension $0\xrightarrow{} A\xrightarrow{\sigma} \LL \xrightarrow{\pi} B\rightarrow{} 0$ and $T$ if $\vp = PT$, $\vp'=P'T$, and $\sigma(f(i,j)) = T(i)T(j) - T(ij)$ for all $i,j\in B$.
\end{defn}

\begin{thm}\label{thm 1}
Given an extension $0\xrightarrow{} A\xrightarrow{\sigma} \LL\xrightarrow{\pi} B\rightarrow{} 0$ of $A$ by $B$ and section $T:B\xrightarrow{} \LL$, there exists a unique factor system $(\vp,\vp',f)$ of $A$ by $B$ belonging to the extension and $T$.
\end{thm}
\begin{proof}
Let $\vp = PT$ and $\vp' = P'T$. To define $f$, one notes that $T(i)T(j) - T(ij)\in \ker \pi$ for any $i,j\in B$. By exactness, there exists an element $c_{i,j}\in A$ such that $\sigma(c_{i,j}) = T(i)T(j) - T(ij)$. Let $f$ be defined by $f(i,j) = c_{i,j}$. One may verify that $f:B\times B\xrightarrow{} A$ is bilinear by applying $\sigma$ to perform the computation and then applying $\sigma\inv$. It remains to verify that $(\vp,\vp',f)$ is a factor system. We compute:
\begin{enumerate}
    \item $\begin{aligned}[t] m(\vp(i)n) &= \sigma\inv(\sigma(m)(T(i)\sigma(n)))\\ &= \sigma\inv((\sigma(m)T(i))\sigma(n) + T(i)(\sigma(m)\sigma(n))) \\ &= (\vp'(i)m)n + \vp(i)(mn), \end{aligned}$
    \item $\begin{aligned}[t] m(\vp'(i)n) &= \sigma\inv(\sigma(m)(\sigma(n)T(i))\\ &= \sigma\inv((\sigma(m)\sigma(n))T(i) + \sigma(n)(\sigma(m)T(i)) \\ &= \vp'(i)(mn) + n(\vp'(i)m), \end{aligned}$
    \item $\begin{aligned}[t]
        mf(i,j) + \vp'(ij)m &= \sigma\inv(\sigma(m)(T(i)T(j)) - \sigma(m)T(ij) +\sigma(m)T(ij)) \\ &= \sigma\inv((\sigma(m)T(i))T(j) + T(i)(\sigma(m)T(j))) \\ &= \sigma\inv\rho'(T(j))\sigma(\sigma\inv\rho'(T(i))\sigma(m)) + \sigma\inv\rho(T(i))\sigma(\sigma\inv\rho'(T(j))\sigma(m)) \\ &= \vp'(j)(\vp'(i)m) + \vp(i)(\vp'(j)m),
    \end{aligned}$
    \item $\begin{aligned}[t] \vp(i)(mn) &= \sigma\inv(T(i) (\sigma(m)\sigma(n))) \\ &= \sigma\inv((T(i)\sigma(m))\sigma(n) +\sigma(m)(T(i)\sigma(n))) \\ &= (\vp(i)m)n + m(\vp(i)n), \end{aligned}$
    \item $\begin{aligned}[t]
        \vp(i)(\vp(j)m) & = \sigma\inv(T(i)(T(j)\sigma(m)))\\ &= \sigma\inv((T(i)T(j))\sigma(m) -T(ij)\sigma(m) + T(j)(T(i)\sigma(m)) + T(ij)\sigma(m)) \\ &= \sigma\inv(\sigma(f(i,j))\sigma(m)) + \sigma\inv\rho(T(j))\sigma(\sigma\inv\rho(T(i))\sigma(m)) + \sigma\inv\rho(T(ij))\sigma(m) \\ &= f(i,j)m + \vp(j)(\vp(i)m) + \vp(ij)m,
    \end{aligned}$
    \item $\begin{aligned}[t] \vp(i)(\vp'(j)m) &= \sigma\inv(T(i)(\sigma(m)T(j))) \\ &= \sigma\inv((T(i)\sigma(m))T(j) + \sigma(m)(T(i)T(j)) + \sigma(m)T(ij) - \sigma(m)T(ij)) \\ &= \sigma\inv\rho'(T(j)\sigma(\sigma\inv\rho(T(i))\sigma(m)) + \sigma\inv\rho'(T(ij))\sigma(m) + \sigma\inv(\sigma(m)\sigma(f(i,j))) \\ &= \vp'(j)(\vp(i)m) + \vp'(ij)m + mf(i,j), \end{aligned}$
    \item $\begin{aligned}[t] f(i,jk) + \vp(i)f(j,k) &= \sigma\inv(T(i)T(jk) - T(i(jk)) + T(i)(T(j)T(k)) - T(i)T(jk)) \\ &= \sigma\inv(T(ij)T(k) - T((ij)k) + (T(i)T(j))T(k) - T(ij)T(k) \\ &~~~~~~~~~~~~ + T(j)T(ik) - T(j(ik)) + T(j)(T(i)T(k)) - T(j)T(ik)) \\ &= \sigma\inv(\sigma(f(ij,k)) + \sigma(f(i,j))T(k) + \sigma(f(j,ik)) + T(j)\sigma(f(i,k))) \\ &= f(ij,k) + \vp'(k)f(i,j) + f(j,ik) + \vp(j)f(i,k). \end{aligned}$
\end{enumerate}
\end{proof}

\begin{thm}\label{thm 2}
(Converse to Theorem \ref{thm 1}) Let $(\vp,\vp',f)$ be a factor system of $A$ by $B$ and let $\LL$ denote the vector space $A\oplus B$ with multiplication $(m,i)(n,j) = (mn + \vp(i)n + \vp'(j)m + f(i,j)~,~ ij)$ for $m,n\in A$ and $i,j\in B$. Let $\sigma:A\xrightarrow{} \LL$ by $\sigma(m) = (m,0)$, $\pi:\LL\xrightarrow{} B$ by $\pi(m,i) = i$, and $T:B\xrightarrow{} \LL$ by $T(i) = (0,i)$. Then \begin{enumerate}
    \item $\LL$ is a Leibniz algebra.
    \item $0\xrightarrow{} A\xrightarrow{\sigma} \LL \xrightarrow{\pi} B \xrightarrow{} 0$ is an extension.
    \item $\pi T = I = I_B$.
    \item The factor system $(\vp,\vp',f)$ belongs to the extension and $T$.
\end{enumerate}
\end{thm}

\begin{proof}
For part 1, the multiplication defined on $\LL$ is clearly linear, and so it suffices to show that the Leibniz identity holds. One computes\footnote{It was from this computation that the axioms of factor systems were chosen.}
\begin{align*}
    (m,i)\big((n,j)(p,k)\big) &= \Big(\underset{\text{Leib.}}{m(np)} + \underset{1.}{m(\vp(j)p)} + \underset{2.}{m(\vp'(k)n)} + \underset{3.}{mf(j,k)} + \underset{\text{4.}}{\vp(i)(np)} + \underset{5.}{\vp(i)(\vp(j)p)}\\ & ~~~~~~~~~~~~~~~~ + \underset{6.}{\vp(i)(\vp'(k)n)} +\underset{7.}{\vp(i)f(j,k)} + \underset{3.}{\vp'(jk)m} + \underset{7.}{f(i,jk)}~,~ \underset{\text{Leib.}}{i(jk)}\Big) \\ &= \Big(\underset{\text{Leib.}}{(mn)p + n(mp)} + \underset{1.}{\underbrace{(\vp'(j)m)p + \vp(j)(mp)}} + \underset{2.}{\underbrace{\vp'(k)(mn) + n(\vp'(k)m)}} \\ &~~~~~~~~~~~~ +\underset{3.}{\underbrace{\vp'(k)(\vp'(j)m) + \vp(j)(\vp'(k)m)}} + \underset{\text{4.}}{\underbrace{(\vp(i)n)p + n(\vp(i)p)}} \\ &~~~~~~~~~~~~ + \underset{5.}{\underbrace{f(i,j)p + \vp(ij)p + \vp(j)(\vp(i)p)}} + \underset{6.}{\underbrace{\vp'(k)(\vp(i)n) + nf(i,k) + \vp'(ik)n}} \\ &~~~~~~~~~~~~ + \underset{7.}{\underbrace{f(ij,k) + \vp'(k)f(i,j) + f(j,ik) + \vp(j)f(i,k)}} ~,~ \underset{\text{Leib.}}{(ij)k + j(ik)} \Big)
\end{align*} where ``Leib." marks the equalities that follow from the Leibniz identities on $A$ and $B$, and numbers $1,\dots,7$ mark the equalities that follow from the axioms of the given factor system. Expanding the sum $((m,i)(n,j))(p,k) + (n,j)((m,i)(p,k))$ via multiplication on $\LL$ yields the same outcome. Hence $\LL$ is a Leibniz algebra.

For part 2, we first compute $\sigma(mn) = (mn,0) = (m,0)(n,0) = \sigma(m)\sigma(n)$ and \begin{align*}
    \pi((m,i)(n,j)) &= \pi(mn + \vp(i)n + \vp'(j)m + f(i,j)~,~ ij) \\ &= ij \\&= \pi(m,i)\pi(n,j)
\end{align*} which implies that $\sigma$ and $\pi$ are homomorphisms. Moreover, the exactness of $0\xrightarrow{} A\xrightarrow{\sigma} \LL\xrightarrow{\pi} B\xrightarrow{} 0$ is trivial. Part 3 is also immediate. For part 4, let $m\in A$ and $i,j\in B$. Then \begin{align*}
    PT(i)m &= \sigma\inv((0,i)(m,0)) \\ &= \sigma\inv(\vp(i)m + f(i,0)~,~ 0) \\ &= \vp(i)m
\end{align*} implies that $PT = \vp$. The equality $P'T = \vp'$ holds by similar computation. Finally, $\sigma(f(i,j)) = (f(i,j),0) = (0,i)(0,j) - (0,ij) = T(i)T(j) - T(ij)$. Hence our factor system belongs to the extension and $T$.
\end{proof}

\subsection{Equivalence}
We now define a relation between factor systems under which a change in $T$ (to an equivalent extension) results in a change in the corresponding factor system to an equivalent factor system. Such a notion establishes equivalence classes of factor systems that correspond to equivalence classes of extensions. Thus Theorem \ref{thm 3} strengthens the correspondence of the first two theorems.

\begin{defn}
Factor systems $(\vp,\vp',f)$ and $(\psi,\psi',g)$ of $A$ by $B$ are called \textit{equivalent} if there exists a linear transformation $E:B\xrightarrow{} A$ such that \begin{enumerate}
    \item $\psi(i) = \vp(i) + \ad^l_{E(i)}$,
    \item $\psi'(i) = \vp'(i) + \ad^r_{E(i)}$,
    \item $g(i,j) = f(i,j) + \vp'(j)E(i) + \vp(i)E(j) + E(i)E(j) - E(ij)$
\end{enumerate} for all $i,j\in B$. The function $E$ is called an \textit{equivalence}.
\end{defn}

\begin{thm} \label{thm 3}
If the factor system $(\vp_1,\vp_1',f_1)$ belongs to the extension $0\xrightarrow{} A\xrightarrow{\sigma_1} \LL_1\xrightarrow{\pi_1} B\rightarrow{} 0$ and $T_1$ and the factor system $(\vp_2,\vp_2',f_2)$ belongs to the extension $0\xrightarrow{} A\xrightarrow{\sigma_2} \LL_2\xrightarrow{\pi_2} B\rightarrow{} 0$ and $T_2$, then the factor systems are equivalent if and only if the extensions are equivalent.
\end{thm}

\begin{proof}
($\implies$) Assume the factor systems are equivalent and let $E$ be the corresponding equivalence. Recall that an equivalence of extensions requires an isomorphism  $\tau:\LL_1\xrightarrow{}\LL_2$ such that $\tau\sigma_1 = \sigma_2$ and $\pi_2\tau = \pi_1$. We know that any element in $\LL_1$ has a unique representation of the form $T_1(i)+\sigma_1(m)$ for $i\in B$ and $m\in A$. Define $\tau(T_1(i)+\sigma_1(m)) = T_2(i) + \sigma_2(-E(i)+m)$. Clearly $\tau$ is linear. To show that $\tau$ preserves multiplication, consider elements $a,b\in \LL_1$ with unique representations $a=T_1(i)+\sigma_1(m)$ and $b=T_1(j)+\sigma_1(n)$. We first compute
\begin{align*}
    \tau(ab) &= \tau\big(T_1(i)T_1(j)+ \sigma_1(m)T_1(j) + T_1(i)\sigma_1(n) + \sigma_1(m)\sigma_1(n)\big)\\
    &= \tau\Big(T_1(ij) + \sigma_1\big(f_1(i,j) + \vp_1'(j)m + \vp_1(i)n + mn\big)\Big) & \text{belonging}\\ &= T_2(ij) + \sigma_2\big(-E(ij) + f_1(i,j) + \vp_1'(j)m + \vp_1(i)n + mn\big).
\end{align*}
On the other hand, one computes \begin{align*}
    \tau(a)\tau(b)&= T_2(i)T_2(j) + T_2(i)\sigma_2(-E(j)) + T_2(i)\sigma_2(n) + \sigma_2(-E(i))T_2(j) \\ & ~~~~+ \sigma_2(m)T_2(j) +\sigma_2(-E(i))\sigma_2(-E(j)) + \sigma_2(m)\sigma_2(-E(j)) \\ & ~~~~ + \sigma_2(-E(i))\sigma_2(n) +\sigma_2(m)\sigma_2(n)\\
    &= T_2(ij) + \sigma_2\big(f_2(i,j) - \vp_2(i)E(j) + \vp_2(i)n - \vp_2'(j)E(i) & \text{belonging}\\ &~~~~~~~~~~~~~~~~~~~~~~~ + \vp_2'(j)m + E(i)E(j) - mE(j) - E(i)n+mn \big) \\ &= T_2(ij) + \sigma_2(p)
\end{align*}
where $p\in A$ is the expression in the argument of $\sigma_2$. Since $T_2(ij)$ is the only $T_2$ term on both sides, it remains to check the $\sigma_2$ parts. Compute \begin{align*}
    p &= f_1(i,j) + \vp_1'(j)E(i) + \vp_1(i)E(j) + E(i)E(j) - E(ij) &\text{equivalence axiom 3} \\
&~~~~ -\vp_1(i)E(j) - E(i)E(j)+ \vp_1(i)n+E(i)n & \text{equivalence axiom 1}\\ &~~~~ - \vp_1'(j)E(i) - E(i)E(j) + \vp_1'(j)m + mE(j) &\text{equivalence axiom 2}\\
    &~~~~ +E(i)E(j) - mE(j) - E(i)n+mn\\
    & = f_1(i,j) - E(ij) + \vp_1(i)n +\vp_1'(j)m + mn.
\end{align*}
Thus $\tau$ preserves multiplication. The computation \begin{align*}
    \pi(T_1(i)+\sigma_1(m)) &= i \\ &= \pi_2(T_2(i) + \sigma_2(-E(i)+m))\\ &= \pi_2\tau(T_1(i)+\sigma_1(m))
\end{align*} implies that $\pi_1 = \pi_2\tau$. Finally, $\tau\sigma_1(m) = \sigma_2(m)$ for all $m\in A$ by the definition of $\tau$. Hence $\tau\sigma_1 = \sigma_2$ and the extensions are equivalent.

$(\impliedby)$ Conversely, assume that the extensions are equivalent. Then there exists an isomorphism $\tau:\LL_1\xrightarrow{} \LL_2$ such that $\tau\sigma_1 = \sigma_2$ and $\pi_2\tau = \pi_1$. The equality $\pi_1\tau\inv T_2(i) = \pi_2T_2(i) = \pi_1 T_1(i)$ holds for any $i\in B$, yielding an element $\tau\inv T_2(i)-T_1(i)\in \ker \pi_1$. By exactness, $\ker \pi_1 = \ima\sigma_1$, and so there exists an element $n_i\in A$ such that $\tau\inv T_2(i) = T_1(i) + \sigma_1(n_i)$. Define $E:B\xrightarrow{} A$ by $E(i) = n_i$. It remains to verify that $E$ is an equivalence. One computes:

\begin{enumerate}
    \item $\begin{aligned}[t] \vp_2(i)n &= P_2T_2(i)n - n_in + n_in\\ &= \sigma_2\inv(T_2(i)\sigma_2(n) - \sigma_2(n_i) \sigma_2(n)) + E(i)n \\ &= \sigma_1\inv\tau\inv((\tau T_1(i) + \tau \sigma_1(n_i))\tau\sigma_1(n) - \tau\sigma_1(n_in)) + E(i)n \\ &= \sigma_1\inv(T_1(i)\sigma_1(n) + \sigma_1(n_in) - \sigma_1(n_in)) + E(i)n \\ &= \vp_1(i)n + E(i)n, \end{aligned}$
    \item $\begin{aligned}[t] \vp_2'(j)m &= P_2'T_2(j)m - mn_j + mn_j \\ &= \sigma_2\inv(\sigma_2(m)T_2(j) - \sigma_2(m)\sigma_2(n_j)) + mE(j) \\ &= \sigma_1\inv\tau\inv(\tau\sigma_1(m)(\tau T_1(j) + \tau\sigma_1(n_j)) - \tau\sigma_1(mn_j)) + mE(j)\\ &= \sigma_1\inv(\sigma_1(m)T_1(j) - \sigma_1(m)\sigma_1(n_j) - \sigma_1(mn_j))) + mE(j)\\ &= \vp_1'(j)m + mE(j), \end{aligned}$
    \item $\begin{aligned}[t] f_2(i,j) &= \sigma_2\inv(T_2(i)T_2(j) - T_2(ij)) \\ &= \sigma_1\inv \tau\inv((\tau T_1(i) + \tau\sigma_1(n_i))(\tau T_1(j) + \tau\sigma_1(n_j)) - \tau T_1(ij) - \tau\sigma_1(n_{ij})) \\ &= \sigma_1\inv(T_1(i)T_1(j) + T_1(i)\sigma_1(n_j) + \sigma_1(n_i)T_1(j) + \sigma_1(n_in_j) - T_1(ij) - \sigma_1(n_{ij})) \\ &= \sigma_1\inv(\sigma_1(f_1(i,j)) + \rho(T_1(i))\sigma_1(n_j) + \rho'(T_1(j))\sigma_1(n_i)) + n_in_j - n_{ij} \\ &= f_1(i,j) + \vp_1'(j)E(i) + \vp_1(i)E(j) + E(i)E(j) - E(ij). \end{aligned}$
\end{enumerate}
Thus the factor systems are equivalent.
\end{proof}

Two results follow easily from Theorem \ref{thm 3}. The proofs are stated as one because they are so short.

\begin{cor}\label{diff Ts}
    Given an extension $0\xrightarrow{} A\xrightarrow{\sigma} \LL\xrightarrow{\pi} B\xrightarrow{} 0$, let $T_1:B\xrightarrow{} \LL$ and $T_2:B\xrightarrow{} \LL$ be linear maps such that $\pi T_1 = I_B = \pi T_2$. Suppose also that $(\vp,\vp',f)$ is a factor system of $A$ by $B$ which belongs to the extension and $T_1$, and $(\psi,\psi',g)$ is a factor system of $A$ by $B$ which belongs to the extension and $T_2$. Then $(\vp,\vp',f)$ is equivalent to $(\psi,\psi',g)$.
\end{cor}

\begin{cor}\label{equiv relation}
Equivalence of factor systems is an equivalence relation.
\end{cor}

\begin{proof}
For Corollary \ref{diff Ts}, note first that any extension of $A$ by $B$ is equivalent to itself. By Theorem \ref{thm 3}, factor systems belonging to this extension (and differing $T_i$) are equivalent. Corollary \ref{equiv relation} follows from Theorem \ref{thm 3} and the fact that equivalence of extensions is an equivalence relation.
\end{proof}

We now look to $E$. Given equivalent factor systems, there may be multiple equivalences between them. On the other hand, \textit{any} linear transformation $E:B\xrightarrow{} A$ defines an equivalence of factor systems, as demonstrated by Theorem \ref{any lin trans}.

\begin{thm} \label{any lin trans}
If $(\vp,\vp',f)$ is a factor system of $A$ by $B$ and $E$ is a linear transformation from $B$ to $A$, then there exists a factor system $(\psi,\psi',g)$ such that $E$ is an equivalence of $(\vp,\vp',f)$ with $(\psi,\psi',g)$. Furthermore, if $E$ is an equivalence, then $(\psi,\psi',g)$ is unique.
\end{thm}

\begin{proof}
Let $(\psi,\psi',g)$ be defined by \begin{enumerate}
    \item[i.] $\psi(i) = \vp(i) + \ad^l_{E(i)}$,
    \item[ii.] $\psi'(j) = \vp'(j) + \ad^r_{E(j)}$,
    \item[iii.] $g(i,j) = f(i,j) + \vp'(j)E(i) + \vp(i)E(j) + E(i)E(j) - E(ij)$
\end{enumerate} for $i,j\in B$. It is straightforward to check that $\psi,\psi':B\xrightarrow{} \mathscr{L}(A)$ are linear transformations and that $g:B\times B\xrightarrow{} A$ is a bilinear form. Now to verify that $(\psi,\psi',g)$ is a factor system.

\begin{enumerate}
    \item $\begin{aligned}[t] m(\psi(i)n) &= m(\vp(i)n) + m(E(i)n) \\ &= (\vp'(i)m)n + \vp(i)(mn) + (mE(i))n + E(i)(mn) & \text{f.s. axiom 1, Leibniz identity} \\ &= (\psi'(i)m)n + \psi(i)(mn), \end{aligned}$
    \item $\begin{aligned}[t] m(\psi'(i)n) &= m(\vp'(i)n) + m(nE(i)) \\ &= \vp'(i)(mn) + n(\vp'(i)m) + (mn)E(i) + n(mE(i)) & \text{f.s. axiom 2, Leibniz identity} \\ &= \psi'(i)(mn) + n(\psi'(i)m), \end{aligned}$
    \item $\begin{aligned}[t] mg(i,j) + \psi'(ij)m &= \underset{3}{\underbrace{mf(i,j)}} + \underset{2}{\underbrace{m(\vp'(j)E(i))}} + \underset{1}{\underbrace{m(\vp(i)E(j))}} + m(E(i)E(j)) - mE(ij) \\ &~~~~~~~~~~~~ +\underset{3}{\underbrace{\vp'(ij)m}} + mE(ij) \\ &= \underset{\text{f.s. axiom 3}}{\underbrace{\vp'(j)(\vp'(i)m) + \vp(i)(\vp'(j)m)}} + \underset{\text{f.s. axiom 2}}{\underbrace{\vp'(j)(mE(i)) + E(i)(\vp'(j)m)}} \\ &~~~~~~~~~~~~ + \underset{\text{f.s. axiom 1}}{\underbrace{(\vp'(i)m)E(j) + \vp(i)(mE(j))}} + \underset{\text{Leibniz}}{\underbrace{(mE(i))E(j) + E(i)(mE(j))}} \\ &= \psi'(j)(\psi'(i)m) + \psi(i)(\psi'(j)m),\end{aligned}$
    \item $\begin{aligned}[t] \psi(i)(mn) &= \vp(i)(mn) + E(i)(mn) \\ &= (\vp(i)m)n + m(\vp(i)n) + (E(i)m)n + m(E(i)n) & \text{f.s. axiom 4, Leibniz} \\ &= (\psi(i)m)n + m(\psi(i)n), \end{aligned}$
    \item $\begin{aligned}[t] \psi(i)(\psi(j)m) &= \underset{5}{\underbrace{\vp(i)(\vp(j)m)}} + \underset{4}{\underbrace{\vp(i)(E(j)m)}} + \underset{1}{\underbrace{E(i)(\vp(j)m)}} + \underset{\text{Leibniz}}{E(i)(E(j)m)}\\ &= \underset{\text{f.s. axiom 5}}{\underbrace{f(i,j)m + \vp(ij)m + \vp(j)(\vp(i)m)}} + \underset{\text{f.s. axiom 4}}{\underbrace{(\vp(i)E(j))m + E(j)(\vp(i)m)}}\\ &~~~~~~~~~~~~ +\underset{\text{f.s. axiom 1}}{\underbrace{(\vp'(i)E(i))m + \vp(j)(E(i)m)}} + \underset{\text{Leibniz}}{(E(i)E(j))m + E(j)(E(i)m)}\\ &~~~~~~~~~~~~ +E(ij)m - E(ij)m \\ &= \psi(ij)m + \psi(j)(\psi(i)m) + g(i,j)m, \end{aligned}$
    \item $\begin{aligned}[t] \psi(i)(\psi'(j)m) &= \underset{6}{\underbrace{\vp(i)(\vp'(j)m)}} + \underset{4}{\underbrace{\vp(i)(mE(j))}} + \underset{2}{\underbrace{E(i)(\vp'(j)m)}} + \underset{\text{Leibniz}}{E(i)(mE(j))} \\ &= \underset{\text{f.s. axiom 6}}{\underbrace{\vp'(j)(\vp(i)m) + mf(i,j) + \vp'(ij)m}} + \underset{\text{f.s. axiom 4}}{\underbrace{(\vp(i)m)E(j) + m(\vp(i)E(j))}} \\ &~~~~~~~~~~~~ +\underset{\text{f.s. axiom 2}}{\underbrace{\vp'(j)(E(i)m) + m(\vp'(j)E(i))}} + \underset{\text{Leibniz}}{(E(i)m)E(j) + m(E(i)E(j))} \\ &~~~~~~~~~~~~ + mE(ij) - mE(ij) \\ &= \psi'(j)(\psi(i)m) + \psi'(ij)m + mg(i,j),\end{aligned}$
    \item $\begin{aligned}[t] g(i,jk) + \psi(i)g(j,k) &= \underset{7}{f(i,jk)} + \underset{3}{\vp'(jk)E(i)} + \vp(i)E(jk) + E(i)E(jk) - \underset{\text{Leibniz}}{E(i(jk))} \\ & ~~~~~~~~~~~~ + \underset{7}{\vp(i)f(j,k)} + \underset{6}{\vp(i)(\vp'(k)E(j))} + \underset{5}{\vp(i)(\vp(j)E(k))} \\ &~~~~~~~~~~~~ + \underset{4}{\vp(i)(E(j)E(k))} - \vp(i)E(jk) + \underset{3}{E(i)f(j,k)} \\&~~~~~~~~~~~~ + \underset{2}{E(i)(\vp'(k)E(j))} + \underset{1}{E(i)(\vp(j)E(k))} + \underset{\text{Leibniz}}{E(i)(E(j)E(k))} \\ &~~~~~~~~~~~~ - E(i)E(jk) \\                                                         &= \underset{7}{f(ij,k)} + \vp'(k)E(ij) + \underset{5}{\vp(ij)E(k)} + E(ij)E(k) - \underset{\text{Leibniz}}{E((ij)k)} \\ &~~~~~~~~~~~~ + \underset{7}{\vp'(k)f(i,j)} + \underset{3}{\vp'(k)(\vp'(j)E(i))} + \underset{6}{\vp'(k)(\vp(i)E(j))} \\ &~~~~~~~~~~~~ +\underset{2}{\vp'(k)(E(i)E(j))} -\vp'(k)E(ij) + \underset{5}{f(i,j)E(k)} \\ &~~~~~~~~~~~~ + \underset{1}{(\vp'(j)E(i))E(k)} + \underset{4}{(\vp(i)E(j))E(k)} + \underset{\text{Leibniz}}{(E(i)E(j))E(k)}\\ &~~~~~~~~~~~~ - E(ij)E(k) +\underset{7}{f(j,ik)} + \underset{6}{\vp'(ik)E(j)} + \vp(j)E(ik) \\ &~~~~~~~~~~~~  + E(j)E(ik) - \underset{\text{Leibniz}}{E(j(ik))} + \underset{7}{\vp(j)f(i,k)} + \underset{3}{\vp(j)(\vp'(k)E(i))} \\ &~~~~~~~~~~~~ +\underset{5}{\vp(j)(\vp(i)E(k))} +\underset{1}{\vp(j)(E(i)E(k))} -\vp(j)E(ik) \\ &~~~~~~~~~~~~ + \underset{6}{E(j)f(i,k)} + \underset{2}{E(j)(\vp'(k)E(i))} + \underset{4}{E(j)(\vp(i)E(k))} \\ &~~~~~~~~~~~~ + \underset{\text{Leibniz}}{E(j)(E(i)E(k))} - E(j)E(ik) \\ &= g(ij,k) + \psi'(k)g(i,j) + g(j,ik) + \psi(j)g(i,k). \end{aligned}$
\end{enumerate}
The unmarked terms cancel in part 7. By construction, the two factor systems are equivalent with $E$ as their corresponding equivalence. It is straightforward to verify the uniqueness of $(\psi,\psi',g)$.
\end{proof}

\subsection{Split Extensions}

Before approaching split extensions, we discuss conditions under which $\vp:B\xrightarrow{} \Der(A)$ is a homomorphism. Let $(\vp,\vp',f)$ be a factor system of $A$ by $B$. By axiom 5 of factor systems, we have $\vp(i)\vp(j) = \vp(ij) + \vp(j)\vp(i) + \ad^l_{f(i,j)}$ for all $i,j\in B$. Then $\vp(ij) = [\vp(i),\vp(j)]$ holds if and only if $f(i,j)\in Z^l(A)$ for all $i,j\in B$. Hence $\vp$ is a homomorphism if and only if $f:B\times B\xrightarrow{} Z^l(A)$. Furthermore, if $Z^l(A)=0$, then $\vp$ is a homomorphism if and only if $f=0$. Finally, if $A$ is abelian, this ensures that $\ad^l_m = 0$ for all $m\in A$. Hence axiom 5 of factor systems again implies that $\vp$ is a homomorphism.

Recall that if an extension splits, then so does every equivalent extension. We thus say that a factor system \textit{splits} if and only if its corresponding extension splits. It follows that if a factor system splits, then so does every equivalent factor system. Now consider a split extension $0\xrightarrow{} A\xrightarrow{\sigma} \LL\xrightarrow{\pi} B\xrightarrow{} 0$ of $A$ by $B$ with associated homomorphism $T:B\xrightarrow{} \LL$ and let $(\vp,\vp',f)$ be a factor system belonging to this extension. Then $\sigma(f(i,j)) = T(i)T(j) - T(ij) = 0$ for all $i,j\in B$ which implies that $f=0$ since $\sigma$ is injective. Axiom 5 of factor systems then implies that $\vp$ is a homomorphism.

The following theorem will be quite useful for later proofs.

\begin{thm}\label{conditions for split}
Let $(\vp,\vp',f)$ be a factor system of $A$ by $B$. The following are equivalent:
\begin{itemize}
    \item[a.] $(\vp,\vp',f)$ splits;
    \item[b.] $(\vp,\vp',f)$ is equivalent to some factor system $(\psi,\psi',g)$ such that $g= 0$;
    \item[c.] there exists a linear transformation $E:B\xrightarrow{} A$ such that $f(i,j) = -\vp'(j)E(i) - \vp(i)E(j) - E(i)E(j) + E(ij)$ for all $i,j\in B$.
\end{itemize}
\end{thm}

\begin{proof}
(a.$\implies$b.) We know $(\vp,\vp',f)$ belongs to a split extension $0\xrightarrow{} A\xrightarrow{\sigma} \LL\xrightarrow{\pi} B\xrightarrow{} 0$. By definition, there is an associated homomorphism $T:B\xrightarrow{} \LL$ such that $\pi T = I_B$. Hence there exists a factor system $(\psi,\psi',g)$ belonging to $0\xrightarrow{} A\xrightarrow{\sigma} \LL\xrightarrow{\pi} B\xrightarrow{} 0$ and $T$ which is equivalent to $(\vp,\vp',f)$ by Corollary \ref{diff Ts}. Since $T$ is a homomorphism, we have $g=0$.

(b.$\implies$c.) Let $E:B\xrightarrow{} A$ be an equivalence of $(\vp,\vp',f)$ with $(\psi,\psi',g)$ where $g=0$. The third axiom of equivalence gives $0=g(i,j) = f(i,j)+ \vp'(j)E(i) + \vp(i)E(j) + E(i)E(j) - E(ij)$ for all $i,j\in B$, which implies the desired equality.

(c.$\implies$a.) Let $E$ be as in c. By Theorem \ref{any lin trans}, $E$ is an equivalence of $(\vp,\vp',f)$ with another factor system $(\psi,\psi',g)$ which belongs to an extension $0\xrightarrow{} A\xrightarrow{\sigma} \LL\xrightarrow{\pi} B\xrightarrow{} 0$ and $T:B\xrightarrow{}\LL$. One has $g(i,j) = f(i,j) + \vp'(j)E(i) + \vp(i)E(j) + E(i)E(j) -E(ij)=0$ by assumption. Then, since $\sigma(g(i,j)) = 0$ for all $i,j\in B$, the third axiom of belonging implies that $T$ is a homomorphism. Also, $T$ is injective since $\pi T = I_B$. Hence the extension splits and, therefore, so does the original factor system.
\end{proof}

It is clear that every semidirect sum yields a split extension. The converse is also true in that every split extension of $A$ by $B$ is equivalent to a semidirect sum. Indeed, let $0\xrightarrow{} A\xrightarrow{} \LL\xrightarrow{} B \xrightarrow{} 0$ be a split extension of $A$ by $B$. By Theorem \ref{conditions for split}, there is an equivalent extension $0\xrightarrow{} A\xrightarrow{} \LL_2\xrightarrow{} B\xrightarrow{} 0$ with associated linear map $T_2:B\xrightarrow{}\LL_2$ and a factor system $(\psi,\psi',g)$ belonging to this extension and $T_2$ such that $g(i,j)= 0$ for all $i,j\in B$. Thus $\psi$ is a homomorphism. The Leibniz algebra construct in Theorem \ref{thm 2} is then a semidirect sum of $A$ by $B$ with factor system $(\psi,\psi',g)$. By Theorem \ref{thm 3}, the extension built in Theorem \ref{thm 2} is equivalent to $0\xrightarrow{}A\xrightarrow{} \LL_2\xrightarrow{} B\xrightarrow{} 0$ and hence to $0\xrightarrow{}A\xrightarrow{} \LL\xrightarrow{} B\xrightarrow{}0$.

\subsection{Abelian $A$}
Let $A$ be an abelian Leibniz algebra and $(\vp,\vp',f)$ be a factor system of $A$ by $B$. Then $\vp$ is a homomorphism. Moreover, suppose a factor system $(\psi,\psi',g)$ of $A$ by $B$ is equivalent to $(\vp,\vp',f)$ via equivalence $E$. Then $\vp(i) = \psi(i) + \ad^l_{E(i)} = \psi(i)$ and $\vp'(i) = \psi'(i) + \ad^r_{E(i)} = \psi'(i)$ for all $i\in B$ which implies that $\vp = \psi$ and $\vp' = \psi'$. We thus fix $\vp$ and $\vp'$ and define the following constructs.

Let $\Fact(B,A,\vp,\vp')$ be the set of bilinear maps $f:B\times B\xrightarrow{} A$ such that $(\vp,\vp',f)$ is a factor system and let $\T(B,A,\vp,\vp')$ be the set of bilinear maps $f:B\times B\xrightarrow{} A$ such that $(\vp,\vp',f)$ is a split factor system. We denote by $\Ext(B,A,\vp,\vp')$ the set of equivalence classes $\Fact(B,A,\vp,\vp')/\T(B,A,\vp,\vp')$ with fixed $\vp$ and $\vp'$.

\begin{thm}\label{abelian algebras}
If $A$ is abelian, then
\begin{enumerate}
    \item $\Fact(B,A,\vp,\vp')$ is an abelian Leibniz algebra,
    \item $\T(B,A,\vp,\vp')$ is an ideal in $\Fact(B,A,\vp,\vp')$,
    \item factor systems $(\vp,\vp',f)$ and $(\vp,\vp',g)$ are equivalent if and only if $f$ and $g$ are in the same coset of $\Fact(B,A,\vp,\vp')$ relative to $\T(B,A,\vp,\vp')$,
    \item the quotient Leibniz algebra $\Ext(B,A,\vp,\vp')$ is in one-to-one correspondence with the set of equivalence classes of extensions to which $\vp$ and $\vp'$ belong.
\end{enumerate}
\end{thm}

\begin{proof}
Let $f,g\in \Fact(B,A,\vp,\vp')$ and $c$ be a scalar. We know $f-cg:B\times B\xrightarrow{} A$ and want to show that $(\vp,\vp',f-cg)$ is a factor system. Axioms 1, 2, and 4 are trivial since they do not involve $f$ or $g$. Axioms 3, 5, and 6 hold since $\ad^l_{(f-cg)(i,j)} = 0$ and $\ad^r_{(f-cg)(i,j)} = 0$ for any $i,j\in B$. Finally, axiom 7 holds by the following computation:\begin{align*}
    (f-cg)(i,jk) + \vp(i)(f-cg)(j,k) & = f(i,jk) + \vp(i)f(j,k) - c(g(i,jk) + \vp(i)g(j,k)) \\ &= f(ij,k) + \vp'(k)f(i,j) + f(j,ik) + \vp(j)f(i,k) \\&~~ - c(g(ij,k) + \vp'(k)g(i,j) + g(j,ik) + \vp(j)g(i,k) ) \\ &= (f-cg)(ij,k) + \vp'(k)(f-cg)(i,j) + (f-cg)(j,ik) \\ &~~ + \vp(j)(f-cg)(i,k).
\end{align*}
Hence $\Fact(B,A,\vp,\vp')$ is a vector space. One easily checks that $(\vp,\vp',fg)$ is a factor system (here juxtaposition denotes $f(i,j)g(i,j)$); indeed, $fg = 0$ since $A$ is abelian. Thus $\Fact(B,A,\vp,\vp')$ is a Leibniz algebra with trivial multiplication.

To show that $\T(B,A,\vp,\vp')$ is an ideal, it suffices to verify that it is a subspace. Let $f,g\in \mathscr{T}(B,A,\vp,\vp')$. We want to show that $(\vp,\vp',f-cg)$ is a split factor system for scalar $c$. By Theorem \ref{conditions for split}, since $(\vp,\vp',f)$ and $(\vp,\vp',g)$ split, there exist linear maps $E_1,E_2:B\xrightarrow{} A$ such that $f(i,j) = -\vp(j)E_1(i) - \vp(i)E_1(j) - E_1(i)E_1(j) + E_1(ij)$ and $g(i,j) = -\vp(j)E_2(i) - \vp(i)E_2(j) - E_2(i)E_2(j) + E_2(ij)$ for any $i,j\in B$. Define $E:B\xrightarrow{} A$ by $E = E_1 - cE_2$. Then $(f-cg)(i,j) = -\vp(j)E(i) + \vp'(i)E(j) - E(i)E(j) + E(ij)$ and so Theorem \ref{conditions for split} says that $(\vp,\vp',f-cg)$ splits.

Suppose factor systems $(\vp,\vp',f)$ and $(\vp,\vp',g)$ are equivalent via $E:B\xrightarrow{} A$. Then $g(i,j) = f(i,j) + \vp'(j)E(i) + \vp(i)E(j) + E(i)E(j) - E(ij)$ implies that $(f-g)(i,j) = -\vp'(j)E(i) - \vp(i)E(j) - E(i)E(j) + E(ij)$. By Theorem \ref{conditions for split}, factor system $(\vp,\vp',f-g)$ splits. Hence $f+\T(B,A,\vp,\vp') = g+\T(B,A,\vp,\vp')$. Conversely, if $(\vp,\vp',f-g)$ is a split factor system, then there exists a linear map $E:B\xrightarrow{} A$ such that $(f-g)(i,j) = -\vp'(j)E(i) - \vp(i)E(j) - E(i)E(j) + E(ij)$ for all $i,j\in B$ (by Theorem \ref{conditions for split}). Thus $E$ satisfies the third axiom of equivalence between factor systems $(\vp,\vp',f)$ and $(\vp,\vp',g)$. The first two axioms of equivalence hold trivially with $\vp=\vp$ and $\vp'=\vp'$ since $\ad^l_m=0$ and $\ad^r_m=0$ for all $m\in A$.

The final statement follows from Theorem \ref{thm 3} and part 3 above. Indeed, part 3 says that two elements of $\Ext(B,A,\vp,\vp')$ are equal if and only if their factor systems with fixed $\vp$, $\vp'$ are equivalent, and Theorem \ref{thm 3} guarantees that the latter statement is true if and only if the two extensions are equivalent.
\end{proof}

\subsection{Central Extensions}
Recall that an extension which is equivalent to a central extension is itself central. One may thus refer to equivalence classes of central extensions and to \textit{central factor systems}, i.e. factor systems that belong to central extensions. Once again, let $A$ be an abelian Leibniz algebra and $(\vp,\vp',f)$ be a factor system of $A$ by $B$.

\begin{thm}\label{central conditions}
$(\vp,\vp',f)$ is central if and only if $\vp=0$ and $\vp'=0$.
\end{thm}

\begin{proof}
By Theorem \ref{thm 2}, $(\vp,\vp',f)$ belongs to an extension $0\xrightarrow{} A\xrightarrow{}\LL \xrightarrow{} B\xrightarrow{} 0$ where $\LL=A\oplus B$ with multiplication $(m,i)(n,j) = (mn+\vp(i)n + \vp'(j)m +f(i,j) ~,~ ij)$. The extension is central if and only if $(m,i)(n,0) = (mn+ \vp(i)n,0) = (0,0)$ and $(n,0)(m,i) = (nm+ \vp'(i)n,0) = (0,0)$ for all $m,n\in A$ and $i\in B$. This happens if and only if $\vp$ and $\vp'$ are zero.
\end{proof}

\begin{thm}
The classes of central extensions of $A$ by $B$ form a Leibniz algebra, denoted $\Cext(B,A)$.
\end{thm}

\begin{proof}
By Theorem \ref{abelian algebras} and Theorem \ref{central conditions}, said classes form the Leibniz algebra $\Ext(B,A,0,0) =: \Cext(B,A)$.
\end{proof}

\begin{thm}
Let $A$ and $B$ be abelian Leibniz algebras and let $(\vp,\vp',f)$ be a central factor system of $A$ by $B$. Then $(\vp,\vp',f)$ belongs to an abelian extension if and only if $f=0$.
\end{thm}

\begin{proof}
Since $(\vp,\vp',f)$ is central, we know $\vp=\vp'=0$. In the forward direction, the factor system belongs to an extension $0\xrightarrow{} A\xrightarrow{} L\xrightarrow{} B\xrightarrow{}0$ and section $T$. Since $L$ and $B$ are both abelian, one has $\sigma(f(i,j)) = T(i)T(j) - T(ij) = 0$ for all $i,j\in B$. Conversely, if $f=0$, then the construction of $L$ in Theorem \ref{thm 2} has multiplication $(m,i)(n,j) = (0,0)$ for all $m,n\in A$ and $i,j\in B$.
\end{proof}

\section{Factor Systems of Diassociative Algebras}
This section mimics the structure of the Leibniz section. We begin by stating the definitions of factor systems (for diassociative algebras and their corresponding special cases) and proceed to construct diassociative analogues of the results from \cite{scott}.

\begin{defn}
Let $A$ and $B$ be diassociative algebras. A \textit{factor system} of $A$ by $B$ is a tuple $(\vp_{\dashv},\vp_{\vdash},\vp_{\dashv}',\vp_{\vdash}',f_{\dashv},f_{\vdash})$ of maps such that $\vp_{\dashv},\vp_{\vdash},\vp_{\dashv}',\vp_{\vdash}':B\xrightarrow{} \mathscr{\LL}(A)$ are linear, $f_{\dashv},f_{\vdash}:B\times B\xrightarrow{} A$ are bilinear, and the following five sets of identities are satisfied for all $m,n,p\in A$ and $i,j,k\in B$:
\begin{enumerate}[topsep=1pt, partopsep=0pt, parsep = 1pt]
    \item Those resembling D1:
\begin{enumerate}[topsep=0pt, noitemsep, partopsep=0pt, parsep=0pt]
    \item $m\dashv(\vp\dd(j)p) = m\dashv(\vp\vv(j)p)$
    \item  $m\dashv(\vp\dd'(k)n) = m\dashv(\vp\vv'(k)n)$
    \item $m\dashv f\dd(j,k) + \vp\dd'(j\dashv k)m = m\dashv f\vv(j,k) + \vp\dd'(j\vdash k)m$
    \item $\vp\dd(i)(n\dashv p) = \vp\dd(i)(n\vdash p)$
    \item $\vp\dd(i)(\vp\dd(j)p) = \vp\dd(i)(\vp\vv(j)p)$
    \item $\vp\dd(i)(\vp\dd'(k)n) = \vp\dd(i)(\vp\vv'(k)n)$
    \item $\vp\dd(i)f\dd(j,k) + f\dd(i,j\dashv k) = \vp\dd(i)f\vv(j,k) + f\dd(i,j\vdash k)$
\end{enumerate}

\item Those resembling D2:
\begin{enumerate}[topsep=0pt, noitemsep, partopsep=0pt, parsep=0pt]
    \item $(\vp\vv(i)n)\dashv p = \vp\vv(i)(n\dashv p)$
    \item $(\vp\vv'(j)m)\dashv p = m\vdash(\vp\dd(j)p)$
    \item $f\vv(i,j)\dashv p + \vp\dd(i\vdash j)p = \vp\vv(i)(\vp\dd(j)p)$
    \item $\vp\dd'(k)(m\vdash n) = m\vdash(\vp\dd'(k)n)$
    \item $\vp\dd'(k)(\vp\vv(i)n) = \vp\vv(i)(\vp\dd'(k)n)$
    \item $\vp\dd'(k)(\vp\vv'(j)m) = \vp\vv'(j\dashv k)m$
    \item $\vp\dd'(k)f\vv(i,j) + f\dd(i\vdash j,k) = \vp\vv(i)f\dd(j,k) + f\vv(i,j\dashv k)$
\end{enumerate}

\item Those resembling D3:
\begin{enumerate}[topsep=0pt, noitemsep, partopsep=0pt, parsep=0pt]
    \item $(\vp\dd(i)n)\vdash p = (\vp\vv(i)n)\vdash p$
    \item $(\vp\dd'(j)m)\vdash p = (\vp\vv'(j)m)\vdash p$
    \item $f\dd(i,j)\vdash p + \vp\vv(i\dashv j)p = f\vv(i,j)\vdash p + \vp\vv(i\vdash j)p$
    \item $\vp\vv'(k)(m\dashv n) = \vp\vv'(k)(m\vdash n)$
    \item $\vp\vv'(k)(\vp\dd(i)n) = \vp\vv'(k)(\vp\vv(i)n)$
    \item $\vp\vv'(k)(\vp\dd'(j)m) = \vp\vv'(k)(\vp\vv'(j)m)$
    \item $\vp\vv'(k)f\dd(i,j) + f\vv(i\dashv j,k) = \vp\vv'(k)f\vv(i,j) + f\vv(i\vdash j,k)$
\end{enumerate}

\item Those resembling the associativity of $\dashv$:
\begin{enumerate}[topsep=0pt, noitemsep, partopsep=0pt, parsep=0pt]
    \item $m\dashv(\vp\dd(j)p) = (\vp\dd'(j)m)\dashv p$
    \item $m\dashv(\vp\dd'(k)n) = \vp\dd'(k)(m\dashv n)$
    \item $m\dashv f\dd(j,k) + \vp\dd'(j\dashv k)m = \vp\dd'(k)(\vp\dd'(j)m)$
    \item $\vp\dd(i)(n\dashv p) = (\vp\dd(i)n)\dashv p$
    \item $\vp\dd(i)(\vp\dd(j)p) = \vp\dd(i\dashv j)p +f\dd(i,j)\dashv p$
    \item $\vp\dd(i)(\vp\dd'(k)n) = \vp\dd'(k)(\vp\dd(i)n)$
    \item $\vp\dd(i)f\dd(j,k) + f\dd(i,j\dashv k) = \vp\dd'(k)f\dd(i,j) + f\dd(i\dashv j,k)$
\end{enumerate}

\item Those resembling the associativity of $\vdash$:
\begin{enumerate}[topsep=0pt, noitemsep, partopsep=0pt, parsep=0pt]
    \item $m\vdash(\vp\vv(j)p) = (\vp\vv'(j)m)\vdash p$
    \item $m\vdash(\vp\vv'(k)n) = \vp\vv'(k)(m\vdash n)$
    \item $m\vdash f\vv(j,k) + \vp\vv'(j\vdash k)m = \vp\vv'(k)(\vp\vv'(j)m)$
    \item $\vp\vv(i)(n\vdash p) = (\vp\vv(i)n)\vdash p$
    \item $\vp\vv(i)(\vp\vv(j)p) = \vp\vv(i\vdash j)p + f\vv(i,j)\vdash p$
    \item $\vp\vv(i)(\vp\vv'(k)n) = \vp\vv'(k)(\vp\vv(i)n)$
    \item $\vp\vv(i)f\vv(j,k) + f\vv(i,j\vdash k) = \vp\vv'(k)f\vv(i,j) + f\vv(i\vdash j,k)$
\end{enumerate}
\end{enumerate}
\end{defn}

\begin{defn}
Let $A$ and $B$ be associative algebras. A \textit{factor system} of $A$ by $B$ is a tuple of maps $(\vp,\vp',f)$ where $\vp,\vp':B\xrightarrow{} \mathscr{L}(A)$ are linear, $f:B\times B\xrightarrow{} A$ is bilinear, and \begin{enumerate}[]
    \item $\vp(i)\vp(j) = \vp(ij) + \ad_{f(i,j)}^l$
    \item $\vp'(i)\vp'(j) = \vp'(ji) + \ad_{f(i,j)}^r$
    \item $\vp(i)\vp'(j) = \vp'(j)\vp(i)$
    \item $\vp'(i)(mn) = m(\vp'(i)n)$
    \item $\vp(i)(mn) = (\vp(i)m)n$
    \item $(\vp'(i)m)n = m(\vp(i)n)$
    \item $\vp(i)f(j,k) + f(i,jk) = \vp'(k)f(i,j) + f(ij,k)$
\end{enumerate} are satisfied for all $m,n\in A$ and $i,j,k\in B$.
\end{defn}

\begin{defn}
Let $A$ and $B$ be commutative algebras. A \textit{factor system} of $A$ by $B$ is a tuple of maps $(\vp,f)$ where $\vp:B\xrightarrow{} \mathscr{L}(A)$ is linear, $f:B\times B\xrightarrow{} A$ is bilinear, and \begin{enumerate}
    \item $f(i,j) = f(j,i)$
    \item $\vp(i)\vp(j) = \vp(ij) + \ad_{f(i,j)}$
    \item $\vp(i)(mn) = m(\vp(i)n) = (\vp(i)m)n$
    \item $\vp(i)f(j,k) + f(i,jk) = \vp(k)f(i,j) + f(ij,k)$
\end{enumerate} are satisfied for all $m,n\in A$ and $i,j,k\in B$.
\end{defn}

\newpage
\begin{defn}
Let $A$ and $B$ be diassociative algebras. A factor system $(\vp\dd,\vp\vv,\vp\dd',\vp\vv',f\dd,f\vv)$ of $A$ by $B$ \textit{belongs} to an extension $0\xrightarrow{} A\xrightarrow{\sigma} \LL\xrightarrow{\pi} B\xrightarrow{} 0$ and section $T$ if
\begin{multicols}{2}
\begin{enumerate}
    \item[] $\vp\dd = P\dd T$,
    \item[] $\vp\dd' = P\dd'T$,
    \item[] $\sigma(f\dd(i,j)) = T(i)\dashv T(j) - T(i\dashv j)$,
    \item[] $\vp\vv = P\vv T$,
    \item[] $\vp\vv' = P\vv'T$,
    \item[] $\sigma(f\vv(i,j)) = T(i)\vdash T(j) - T(i\vdash j)$
\end{enumerate}
\end{multicols}
\noindent for all $i,j\in B$, where \begin{multicols}{2} \begin{enumerate}
    \item[] $P\dd(x)m = \sigma\inv(x\dashv\sigma(m))$,
    \item[] $P\dd'(x)m = \sigma\inv(\sigma(m)\dashv x)$,
    \item[] $P\vv(x)m = \sigma\inv(x\vdash\sigma(m))$,
    \item[] $P\vv'(x)m = \sigma\inv(\sigma(m)\vdash x)$
\end{enumerate} \end{multicols} \noindent for any $x\in L$, $m\in A$.
\end{defn}

\begin{thm} Let $A$ and $B$ be diassociative algebras. Given an extension $0\xrightarrow{} A\xrightarrow{} L\xrightarrow{\pi} B\xrightarrow{} 0$ of $A$ by $B$ and section $T:B\xrightarrow{} L$, there exists a unique factor system $(\vp\dd,\vp\vv,\vp\dd',\vp\vv',f\dd,f\vv)$ of $A$ by $B$ belonging to the extension and $T$.
\end{thm}

\begin{proof}
Set $\vp\dd = P\dd T$, $\vp\vv = P\vv T$, $\vp\dd' = P\dd' T$, and $\vp\vv' = P\vv' T$. Next, it is easily checked that $T(i)\dashv T(j) - T(i\dashv j)$ and $T(i)\vdash T(j) - T(i\vdash j)$ are in the kernel of $\pi$ for all $i,j\in B$. We define $f\dd$ and $f\vv$ by $\sigma(f\dd(i,j)) = T(i)\dashv T(j) - T(i\dashv j)$ and $\sigma(f\vv(i,j)) = T(i)\vdash T(j) - T(i\vdash j)$ which are clearly bilinear maps $B\times B\xrightarrow{} A$. It is straightforward to verify that $(\vp\dd,\vp\vv,\vp\dd',\vp\vv',f\dd,f\vv)$ is a factor system.
\end{proof}

\begin{thm}\label{di thm 2}
Let $(\vp\dd,\vp\vv,\vp\dd',\vp\vv',f\dd,f\vv)$ be a factor system of $A$ by $B$ and let $L$ denote the vector space $A\oplus B$ with multiplications \begin{enumerate}
    \item[] $(m,i)\vdash (n,j) = (m\vdash n + \vp\vv(i)n + \vp\vv'(j)m + f\vv(i,j) ~,~ i\vdash j)$,
    \item[] $(m,i)\dashv (n,j) = (m\dashv n + \vp\dd(i)n + \vp\dd'(j)m + f\dd(i,j)~ , ~i\dashv j)$
\end{enumerate} for $m,n\in A$ and $i,j\in B$. Let $\sigma:A\xrightarrow{} L$ by $\sigma(m) = (m,0)$, $\pi:L\xrightarrow{} B$ by $\pi(m,i) = i$, and $T:B\xrightarrow{} L$ by $T(i) = (0,i)$. Then \begin{enumerate}
    \item $L$ is a diassociative algebra.
    \item $0\xrightarrow{} A\xrightarrow{\sigma} L\xrightarrow{\pi} B\xrightarrow{} 0$ is an extension.
    \item $\pi T = I = I_B$.
    \item The factor system $(\vp\dd,\vp\vv,\vp\dd',\vp\vv',f\dd,f\vv)$ belongs to the extension and $T$.
\end{enumerate}
\end{thm}

\begin{proof}
It takes five direct computations to verify that the vector space $L=A\oplus B$, with multiplications defined in the statement of the theorem, is a diassociative algebra. In particular, one must check D1, D2, D3, and the associativity of both $\dashv$ and $\vdash$. Said computations follow via the axioms of factor systems and the diassociative structures on $A$ and $B$.
\end{proof}

We now define a notion of equivalence for factor systems so that equivalence classes of factor systems will correspond to those of extensions. The subsequent corollaries hold by the same logic as their Leibniz analogues.

\begin{defn}
Two factor systems $(\vp\dd,\vp\vv,\vp\dd',\vp\vv',f\dd,f\vv)$ and $(\psi\dd,\psi\vv,\psi\dd',\psi\vv',g\dd,g\vv)$ of $A$ by $B$ are \textit{equivalent} if there exists a linear transformation $E:B\xrightarrow{} A$ such that
\begin{enumerate}
    \item $\psi\dd(i) = \vp\dd(i) + \ad\dd^l(E(i))$,
    \item $\psi\dd'(i) = \vp\dd'(i) + \ad\dd^r(E(i))$,
    \item $\psi\vv(i) = \vp\vv(i) + \ad\vv^l(E(i))$,
    \item $\psi\vv'(i) = \vp\vv'(i) + \ad\vv^r(E(i))$,
    \item $g\dd(i,j) = f\dd(i,j) + \vp\dd'(j)E(i) + \vp\dd(i)E(j) + E(i)\dashv E(j) - E(i\dashv j)$,
    \item $g\vv(i,j) = f\vv(i,j) + \vp\vv'(j)E(i) + \vp\vv(i)E(j) + E(i)\vdash E(j) - E(i\vdash j)$
\end{enumerate}
for all $i,j\in B$  where $\ad\dd^l(E(i))m = E(i)\dashv m$, $\ad\vv^l(E(i))m = E(i)\vdash m$, $\ad\dd^r(E(i))m = m\dashv E(i)$, and $\ad\vv^r(E(i))m = m\vdash E(i)$ for all $m\in A$.
\end{defn}

\begin{thm}\label{di thm 3}
If the factor system $(\vp\dd,\vp\vv,\vp\dd', \vp\vv',f\dd,f\vv)$ belongs to the extension $0\xrightarrow{} A\xrightarrow{\sigma_1} \LL_1\xrightarrow{\pi_1} B\rightarrow{} 0$ and $T_1$ and the factor system $(\psi\dd,\psi\vv,\psi\dd',\psi\vv',g\dd,g\vv)$ belongs to the extension $0\xrightarrow{} A\xrightarrow{\sigma_2} \LL_2\xrightarrow{\pi_2} B\rightarrow{} 0$ and $T_2$, then the factor systems are equivalent if and only if the extensions are equivalent.
\end{thm}

\begin{proof}
In the forward direction, one defines $\tau$ in the same way as the Leibniz case and computes $\tau(a\dashv b) = \tau(a)\dashv \tau(b)$ and $\tau(a\vdash b) = \tau(a)\vdash \tau(b)$ via the axioms of equivalence for diassociative factor systems. In the other direction, one defines $E(i) =n_i$ where $\tau\inv T_2(i) = T_1(i) + \sigma_1(n_i)$. There are six axioms to check when verifying that $E$ is an equivalence of factor systems. Otherwise, the theorem follows by similar logic.
\end{proof}

\begin{cor}\label{di diff Ts}
    Given an extension $0\xrightarrow{} A\xrightarrow{\sigma} \LL\xrightarrow{\pi} B\xrightarrow{} 0$, let $T_1:B\xrightarrow{} \LL$ and $T_2:B\xrightarrow{} \LL$ be linear maps such that $\pi T_1 = I_B = \pi T_2$. Suppose also that $(\vp\dd,\vp\vv,\vp\dd', \vp\vv',f\dd,f\vv)$ is a factor system of $A$ by $B$ which belongs to the extension and $T_1$, and $(\psi\dd,\psi\vv,\psi\dd',\psi\vv',g\dd,g\vv)$ is a factor system of $A$ by $B$ which belongs to the extension and $T_2$. Then $(\vp\dd,\vp\vv,\vp\dd', \vp\vv',f\dd,f\vv)$ is equivalent to $(\psi\dd,\psi\vv,\psi\dd',\psi\vv',g\dd,g\vv)$.
\end{cor}

\begin{cor}\label{di equiv relation}
Equivalence of factor systems is an equivalence relation.
\end{cor}

\begin{thm}\label{di any lin trans}
If $(\vp\dd,\vp\vv,\vp\dd', \vp\vv',f\dd,f\vv)$ is a factor system of $A$ by $B$ and $E$ is a linear transformation from $B$ to $A$, then there exists a factor system $(\psi\dd,\psi\vv,\psi\dd',\psi\vv',g\dd,g\vv)$ such that $E$ is an equivalence between them. Furthermore, if $E$ is an equivalence, then $(\psi\dd,\psi\vv,\psi\dd',\psi\vv',g\dd,g\vv)$ is unique.
\end{thm}

\begin{proof}
Define \begin{enumerate}[topsep=0pt, itemsep=1pt]
    \item[i.] $\psi\dd(i) = \vp\dd(i) + \ad\dd^l(E(i))$,
    \item[ii.] $\psi\vv(i) = \vp\vv(i) + \ad\vv^l(E(i))$,
    \item[iii.] $\psi\dd'(i) = \vp\dd'(i) + \ad\dd^r(E(i))$,
    \item[iv.] $\psi\vv'(i) = \vp\vv'(i) + \ad\vv^r(E(i))$,
    \item[v.] $g\dd(i,j) = f\dd(i,j) + \vp\dd'(j)E(i) + \vp\dd(i)E(j) + E(i)\dashv E(j) - E(i\dashv j)$,
    \item[vi.] $g\vv(i,j) = f\vv(i,j) + \vp\vv'(j)E(i) + \vp\vv(i)E(j) + E(i)\vdash E(j) - E(i\vdash j)$
\end{enumerate} for all $i,j\in B$. It is straightforward to verify that $\psi\dd$, $\psi\vv$, $\psi\dd'$, and $\psi\vv'$ are linear transformations and that $g\dd$ and $g\vv$ are bilinear forms. One checks that $(\psi\dd,\psi\vv,\psi\dd',\psi\vv',g\dd,g\vv)$ is a factor system via the identities of $(\vp\dd,\vp\vv,\vp\dd', \vp\vv',f\dd,f\vv)$ and the axioms of diassociative algebras. By construction, the two factor systems are equivalent with $E$ as their corresponding equivalence. It is straightforward to verify uniqueness.
\end{proof}

\begin{thm}\label{di split cond}
Let $(\vp\dd,\vp\vv,\vp\dd',\vp\vv',f\dd,f\vv)$ be a factor system of $A$ by $B$. The following are equivalent:
\begin{enumerate}[topsep=0pt]
    \item[a.] $(\vp\dd,\vp\vv,\vp\dd',\vp\vv',f\dd,f\vv)$ splits;
    \item[b.] $(\vp\dd,\vp\vv,\vp\dd',\vp\vv',f\dd,f\vv)$ is equivalent to some factor system $(\psi\dd,\psi\vv,\psi\dd',\psi\vv',g\dd,g\vv)$ such that $g\dd = 0$ and $g\vv = 0$;
    \item[c.] there exists a linear transformation $E:B\xrightarrow{} A$ such that \begin{enumerate}[itemsep=0pt, topsep=0pt]
        \item[] $f\dd(i,j) = -\vp\dd'(j)E(i) - \vp\dd(i)E(j) - E(i)\dashv E(j) + E(i\dashv j)$,
        \item[] $f\vv(i,j) = -\vp\vv'(j)E(i) - \vp\vv(i)E(j) - E(i)\vdash E(j) + E(i\vdash j)$.
    \end{enumerate}
\end{enumerate}
\end{thm}

\begin{proof}
(a.$\implies$b.) We know $(\vp\dd,\vp\vv,\vp\dd',\vp\vv',f\dd,f\vv)$ belongs to a split extension $0\xrightarrow{} A\xrightarrow{\sigma} L\xrightarrow{\pi} B\xrightarrow{}0$. By definition, there is an associated homomorphism $T:B\xrightarrow{} L$ such that $\pi T = I_B$. Hence there exists a factor system $(\psi\dd,\psi\vv,\psi\dd',\psi\vv',g\dd,g\vv)$ belonging to the extension and $T$ which is equivalent to $(\vp\dd,\vp\vv,\vp\dd',\vp\vv',f\dd,f\vv)$ by Corollary \ref{di diff Ts}. Since $T$ is a homomorphism, we have $g\dd=g\vv=0$.

(b.$\implies$c.) Let $E:B\xrightarrow{} A$ be an equivalence of factor systems $(\vp\dd,\vp\vv,\vp\dd',\vp\vv',f\dd,f\vv)$ and $(\psi\dd,\psi\vv,\psi\dd',\psi\vv',g\dd,g\vv)$ where $g\dd=g\vv=0$. Then $0= g\dd(i,j) = f\dd(i,j) + \vp\dd'(j)E(i) + \vp\dd(i)E(j) + E(i)\dashv E(j) - E(i\dashv j)$ and $0= g\vv(i,j) = f\vv(i,j) + \vp\vv'(j)E(i) + \vp\vv(i)E(j) + E(i)\vdash E(j) - E(i\vdash j)$ for all $i,j\in B$ by the axioms of equivalence, which implies the desired equalities.

(c.$\implies$a.) Let $E$ be as in c. By Theorem \ref{di any lin trans}, $E$ is an equivalence of $(\vp\dd,\vp\vv,\vp\dd',\vp\vv',f\dd,f\vv)$ with another factor system $(\psi\dd,\psi\vv,\psi\dd',\psi\vv',g\dd,g\vv)$ which belongs to an extension $0\xrightarrow{} A\xrightarrow{\sigma} L\xrightarrow{\pi} B\xrightarrow{}0$ and $T:B\xrightarrow{} L$. One has $g\dd(i,j) = f\dd(i,j) + \vp\dd'(j)E(i) + \vp\dd(i)E(j) + E(i)\dashv E(j) - E(i\dashv j) = 0$ and $g\vv(i,j) = f\vv(i,j) + \vp\vv'(j)E(i) + \vp\vv(i)E(j) + E(i)\vdash E(j) - E(i\vdash j) = 0$ by assumption. Then, since $\sigma(g(i,j)) = 0$ for all $i,j\in B$, the axioms of belonging imply that $T$ is a homomorphism. Also, $T$ is injective since $\pi T= I_B$. Hence the extension splits and, therefore, so does the original factor system.
\end{proof}

Let $A$ be an abelian diassociative algebra and let $(\vp\dd,\vp\vv,\vp\dd',\vp\vv',f\dd,f\vv)$ be a factor system of $A$ by $B$ which is equivalent to another factor system $(\psi\dd,\psi\vv,\psi\dd',\psi\vv',g\dd,g\vv)$. Since $A$ is abelian, all adjoint operators on $A$ are equal to zero. Thus, by the axioms of equivalence for factor systems, $\vp\dd = \psi\dd$, $\vp\vv=\psi\vv$, $\vp\dd' = \psi\dd'$, and $\vp\vv'=\psi\vv'$. We thus fix the first four maps of factor systems and narrow our focus to pairs of bilinear forms. Let $\Fact(B,A,\vp\dd,\vp\vv,\vp\dd',\vp\vv')$ denote the set of all pairs $(f\dd,f\vv)$ such that $(\vp\dd,\vp\vv,\vp\dd',\vp\vv',f\dd,f\vv)$ is a factor system and let $\T(B,A,\vp\dd,\vp\vv,\vp\dd',\vp\vv')$ denote the set of all pairs $(f\dd,f\vv)$ such that $(\vp\dd,\vp\vv,\vp\dd',\vp\vv',f\dd,f\vv)$ is a split factor system. For ease of notation, let $\vp$ denote the fixed tuple $(\vp\dd,\vp\vv,\vp\dd',\vp\vv')$ and let \[(\vp,f\dd,f\vv)\] denote the factor system $(\vp\dd,\vp\vv,\vp\dd',\vp\vv',f\dd,f\vv)$. We abbreviate the previous sets by $\Fact_{\vp}$ and $\T_{\vp}$ respectively and denote by \[\Ext(B,A,\vp\dd,\vp\vv,\vp\dd',\vp\vv')\] the set of equivalence classes $\Fact_{\vp}/\T_{\vp}$. For the rest of this section, we assume $A$ is abelian.

\begin{thm}\label{di abelian algebras}
If $A$ is abelian, then
\begin{enumerate}
    \item $\Fact(B,A,\vp\dd,\vp\vv,\vp\dd',\vp\vv')$ is an abelian diassociative algebra,
    \item $\T(B,A,\vp\dd,\vp\vv,\vp\dd',\vp\vv')$ is an ideal in $\Fact(B,A,\vp\dd,\vp\vv,\vp\dd',\vp\vv')$,
    \item factor systems $(\vp\dd,\vp\vv,\vp\dd',\vp\vv',f\dd,f\vv)$ and $(\vp\dd,\vp\vv,\vp\dd',\vp\vv',g\dd,g\vv)$ are equivalent if and only if $(f\dd,f\vv)$ and $(g\dd,g\vv)$ are in the same coset of $\Fact_{\vp}$ relative to $\T_{\vp}$,
    \item the quotient diassociative algebra $\Ext(B,A,\vp\dd,\vp\vv,\vp\dd',\vp\vv')$ is in one-to-one correspondence with the set of equivalence classes of extensions to which $\vp\dd$, $\vp\vv$, $\vp\dd'$, and $\vp\vv'$ belong.
\end{enumerate}
\end{thm}

\begin{proof}
For $(f\dd,f\vv)$ and $(g\dd,g\vv)$ in $\Fact_{\vp}$, one verifies $(f\dd - cg\dd, f\vv-cg\vv)\in \Fact_{\vp}$ via the axioms of the factor systems $(\vp,f\dd,f\vv)$ and $(\vp,g\dd,g\vv)$ and the fact that multiplication in $A$ is trivial. For the second statement, it suffices to verify that $\T_{\vp}$ is a subspace. Consider elements $(f\dd,f\vv)$ and $(g\dd,g\vv)$ in $\T_{\vp}$, which form split factor systems $(\vp,f\dd,f\vv)$ and $(\vp,g\dd,g\vv)$ respectively. By Theorem \ref{di split cond}, there exist linear transformations $E_f,E_g:B\xrightarrow{} A$ such that \begin{enumerate}[itemsep=2pt]
    \item[] $f\dd(i,j) = -\vp\dd'(j)E_f(i) - \vp\dd(i)E_f(j) - E_f(i)\dashv E_f(j) + E_f(i\dashv j)$,
    \item[] $f\vv(i,j) = -\vp\vv'(j)E_f(i) - \vp\vv(i)E_f(j) - E_f(i)\vdash E_f(j) + E_f(i\vdash j)$,
\end{enumerate}
and
\begin{enumerate}[itemsep=2pt]
    \item[] $g\dd(i,j) = -\vp\dd'(j)E_g(i) - \vp\dd(i)E_g(j) - E_g(i)\dashv E_g(j) + E_g(i\dashv j)$,
    \item[] $g\vv(i,j) = -\vp\vv'(j)E_g(i) - \vp\vv(i)E_g(j) - E_g(i)\vdash E_g(j) + E_g(i\vdash j)$.
\end{enumerate} Letting $E = E_f-cE_g$, one has \begin{enumerate}[itemsep=2pt]
    \item[] $(f\dd - cg\dd)(i,j) = -\vp\dd'(j)E(i) -\vp\dd(i)E(j) - E(i)\dashv E(j) + E(i\dashv j)$,
    \item[] $(f\vv - cg\vv)(i,j) = -\vp\vv'(j)E(i) -\vp\vv(i)E(j) - E(i)\vdash E(j) + E(i\vdash j)$
\end{enumerate} which implies that $(\vp,f\dd-cg\dd, f\vv-cg\vv)$ splits. For the third statement, one observes that the last two axioms of equivalence for factor systems hold if and only if the third condition of Theorem \ref{di split cond} holds for the factor system $(\vp,f\dd-g\dd, f\vv-g\vv)$. Since $A$ is abelian, adjoint operators on $A$ are trivial. The final statement holds as in the Leibniz analogue.
\end{proof}

\begin{thm}\label{di central conditions}
$(\vp\dd,\vp\vv,\vp\dd',\vp\vv',f\dd, f\vv)$ is central if and only if $\vp\dd,\vp\vv,\vp\dd',\vp\vv'=0$.
\end{thm}

\begin{proof}
By Theorem \ref{di thm 2}, the factor system belongs to an extension $0\xrightarrow{} A\xrightarrow{} L\xrightarrow{} B\xrightarrow{} 0$ which is central if and only if $(m,i)\dashv(n,0) = (n,0)\dashv(m,i)= (m,i)\vdash(n,0) = (n,0)\vdash(m,i) = (0,0)$ for all $m,n\in A$ and $i\in B$. But this happens if and only if $\vp\dd,\vp\vv,\vp\dd',\vp\vv'=0$.
\end{proof}

\begin{thm}
The classes of central extensions of $A$ by $B$ form a diassociative algebra, denoted $\Cext(B,A)$.
\end{thm}

\begin{proof}
By Theorem \ref{di abelian algebras} and Theorem \ref{di central conditions}, said classes form a diassociative algebra $\Cext(B,A):= \Ext(B,A,0,0,0,0)$.
\end{proof}

\begin{thm}
Let $A$ and $B$ be abelian diassociative algebras and let $(\vp,f\dd, f\vv)$ be a central factor system of $A$ by $B$. Then $(\vp,f\dd, f\vv)$ belongs to an abelian extension if and only if $f\dd=0$ and $f\vv=0$.
\end{thm}

\begin{proof}
Since $(\vp,f\dd, f\vv)$ is central, we know all $\vp$ maps are zero. In the forward direction, $(\vp,f\dd, f\vv)$ belongs to an abelian extension $0\xrightarrow{} A\xrightarrow{} L\xrightarrow{} B\xrightarrow{} 0$ and section $T$. Since $L$ and $B$ are both abelian, one has $\sigma(f\dd(i,j)) = T(i)\dashv T(j) - T(i\dashv j) = 0$ and $\sigma(f\vv(i,j)) = T(i)\vdash T(j) - T(i\vdash j) = 0$ for all $i,j\in B$. Conversely, if $f\dd$ and $f\vv$ are trivial, then the construction of $L$ in Theorem \ref{di thm 2} has trivial multiplications.
\end{proof}

\section{Factor Systems of Dendriform Algebras}

The dendriform versions of these results follow by the same logic as the diassociative case with the substitutions of $<$ and $>$ for multiplications $\dashv$ and $\vdash$ respectively. We note again that Zinbiel algebras are a special case of dendriform algebras, and the identities of factor systems condense along these lines.

\begin{defn}
Let $A$ and $B$ be dendriform algebras. A \textit{factor system} of $A$ by $B$ is a tuple $(\vp_<,\vp_>,\vp_<',\vp_>',f_<,f_>)$ of maps such that $\vp_<,\vp_>,\vp_<',\vp_>':B\xrightarrow{} \mathscr{\LL}(A)$ are linear, $f_<,f_>:B\times B\xrightarrow{} A$ are bilinear, and the following three sets of identities are satisfied for all $m,n,p\in A$ and $i,j,k\in B$:
\begin{enumerate}[topsep=1pt, partopsep=0pt, parsep = 1pt]
    \item Those resembling E1:
\begin{enumerate}[topsep=0pt, noitemsep, partopsep=0pt, parsep=0pt]
    \item $(\vp_<(i)n)<p = \vp_<(i)(n<p) + \vp_<(i)(n>p)$
    \item $(\vp_<'(j)m)<p = m<(\vp_<(j)p) + m<(\vp_>(j)p)$
    \item $f_<(i,j)<p + \vp_<(i<j)p = \vp_<(i)(\vp_<(j)p) + \vp_<(i)(\vp_>(j)p)$
    \item $\vp_<'(k)(m<n) = m<(\vp_<'(k)n) + m<(\vp_>'(k)n)$
    \item $\vp_<'(k)(\vp_<(i)n) = \vp_<(i)(\vp_<'(k)n) + \vp_<(i)(\vp_>'(k)n)$
    \item $\vp_<'(k)(\vp_<'(j)m) = m<f_<(j,k) + \vp_<'(j<k)m + m<f_>(j,k) + \vp_<'(j>k)m$
    \item $\vp_<'(k)f_<(i,j) + f_<(i<j,k) = \vp_<(i)f_<(j,k) + f_<(i,j<k) + \vp_<(i)f_>(j,k) + f_<(i,j>k)$
\end{enumerate}

\item Those resembling E2:
\begin{enumerate}[topsep=0pt, noitemsep, partopsep=0pt, parsep=0pt]
    \item $(\vp_>(i)n)<p = \vp_>(i)(n<p)$
    \item $(\vp_>'(j)m)<p = m>(\vp_<(j)p)$
    \item $f_>(i,j)<p + \vp_<(i>j)p = \vp_>(i)(\vp_<(j)p)$
    \item $\vp_<'(k)(m>n) = m>(\vp_<'(k)n)$
    \item $\vp_<'(k)(\vp_>(i)n) = \vp_>(i)(\vp_<'(k)n)$
    \item $\vp_<'(k)(\vp_>'(j)m) = \vp_>'(j<k)m + m>f_<(j,k)$
    \item $\vp_<'(k)f_>(i,j) + f_<(i>j,k) = \vp_>(i)f_<(j,k) + f_>(i,j<k)$
\end{enumerate}

\item Those resembling E3:
\begin{enumerate}[topsep=0pt, noitemsep, partopsep=0pt, parsep=0pt]
    \item $(\vp_<(i)n)>p + (\vp_>(i)n)>p = \vp_>(i)(n>p)$
    \item $(\vp_<'(j)m)>p + (\vp_>'(j)m)>p = m>(\vp_>(j)p)$
    \item $f_<(i,j)>p + \vp_>(i<j)p + f_>(i,j)>p + \vp_>(i>j)p = \vp_>(i)(\vp_>(j)p)$
    \item $\vp_>'(k)(m<n) + \vp_>'(k)(m>n) = m>(\vp_>'(k)n)$
    \item $\vp_>'(k)(\vp_<(i)n) + \vp_>'(k)(\vp_>(i)n) = \vp_>(i)(\vp_>'(k)n)$
    \item $\vp_>'(k)(\vp_<'(j)m) + \vp_>'(k)(\vp_>'(j)m) = m>f_>(j,k) + \vp_>'(j>k)m$
    \item $\vp_>'(k)f_<(i,j) + f_>(i<j,k) + \vp_>'(k)f_>(i,j) + f_>(i>j,k) = \vp_>(i)f_>(j,k) + f_>(i,j>k)$
\end{enumerate}
\end{enumerate}
\end{defn}

\begin{defn}
Let $A$ and $B$ be Zinbiel algebras. A \textit{factor system} of $A$ by $B$ is a tuple of maps $(\vp,\vp',f)$ where $\vp,\vp':B\xrightarrow{} \mathscr{L}(A)$ are linear, $f:B\times B\xrightarrow{} A$ is bilinear, and
\begin{enumerate}[itemsep=1pt]
    \item $(\vp(i)n)p = \vp(i)(np) + \vp(i)(pn)$
    \item $(\vp'(j)m)p = m(\vp(j)p) + m(\vp'(j)p)$
    \item $\vp'(k)(mn) = m(\vp(k)n) + m(\vp'(k)n)$
    \item $f(i,j)p + \vp(ij)p = \vp(i)(\vp(j)p)) + \vp(i)(\vp'(j)p)$
    \item $\vp'(k)(\vp(i)n) = \vp(i)(\vp(k)n) + \vp(i)(\vp'(k)n)$
    \item $\vp'(k)(\vp'(j)m) = mf(j,k) + mf(k,j) + \vp'(jk)m + \vp'(kj)m$
    \item $\vp'(k)f(i,j) + f(ij,k) = \vp(i)f(j,k) + \vp(i)f(k,j) + f(i,jk) + f(i,kj)$
\end{enumerate} are satisfied for all $m,n,p\in A$ and $i,j,k\in B$.
\end{defn}

\section{Cohomology}
Given a central extension $0\xrightarrow{} A\xrightarrow{} L\xrightarrow{} B\xrightarrow{} 0$ of $A$ by $B$, the general construction of cohomology for Leibniz algebras begins with the set $\cC^n(B,A)$ of $n$-linear maps $f:B\times\cdots \times B\xrightarrow{} A$. Elements of $\cC^n(B,A) = \text{Mult}(B\times \cdots\times B, A) \cong\Hom_{\F}(B^{\otimes n},A)$ are called \textit{$n$-cochains}. The usual Leibniz coboundary map $d^n:\cC^n(B,A)\xrightarrow{} \cC^{n+1}(B,A)$ simplifies (in the central case) to: \begin{align*}
    (d^nf)(x_1,\dots,x_{n+1}) = \sum_{1\leq i<j\leq n+1} (-1)^i f(x_1,\dots,\hat{x_i},\dots,x_{j-1},x_ix_j,x_{j+1}, \dots,x_{n+1})
\end{align*} for $f\in \cC^n(B,A)$. Note specifically that $(d^2f)(x,y,z) = - f(xy,z) + f(x,yz) - f(y,xz)$. We denote by $\cZ^n(B,A)$ the set of all $f\in \cC^n(B,A)$ such that $d^nf = 0$ (\textit{$n$-cocycles}) and by $\cB^n(B,A)$ the set of all $f\in \cC^n(B,A)$ such that $d^{n-1}E = f$ for some $E\in \cC^{n-1}(B,A)$ (\textit{$n$-coboundaries}). It is well known that $d^{n}\circ d^{n-1}=0$ and thus $\cB^n(B,A)\subseteq \cZ^n(B,A)$. Therefore $\cH^n(B,A) = \cZ^n(B,A)/\cB^n(B,A)$ is the \textit{$n$th cohomology group}. In particular, $\cH^2(B,A)$ arises from our theory of extensions and factor systems. First recall Theorem \ref{abelian algebras} and the construction $\Fact(B,A,0,0)$. Given a central extension $0\xrightarrow{} A\xrightarrow{} L\xrightarrow{} B\xrightarrow{} 0$, the axioms of the corresponding factor systems become trivial except for the seventh one, which reduces to $f(i,jk) = f(ij,k) + f(j,ik)$, the \textit{2-cocycle identity} for Leibniz algebras. Thus $\Fact(B,A,0,0)$ is the set of all bilinear maps $f:B\times B\xrightarrow{} A$ such that $d^2f = 0$. By Theorem \ref{conditions for split}, $\mathscr{T}(B,A,0,0)$ is the set of all bilinear $f:B\times B\xrightarrow{} A$ such that $f(i,j)=-E(ij)$ for some linear transformation $E:B\xrightarrow{} A$. These sets are thereby the 2-cocycles and 2-coboundaries of our cohomology respectively. In other words, $\cZ^2(B,A) = \Fact(B,A,0,0)$ and $\cB^2(B,A) = \T(B,A,0,0)$. Thus $\Cext(B,A)$ is the second cohomology group $\cH^2(B,A)$. We refer the reader to Loday's \cite{loday dialgebras} for constructions of (co)homology in the diassociative and dendriform settings.

In conclusion, each set of 2-cocycle identities is the central simplification of the identities for factor systems. The following table lists these identities for each $\alg$ algebra, as well as the total numbers $\mu(\alg)$ of factor system identities. By construction, each cocycle formulation resembles the defining identities of the corresponding $\alg$ structure.
\bigskip

\begin{tabular}{c|c|c|c}
$\alg$ & $\mu(\alg)$ & 2-cocycle form & 2-cocycle Identities \\ \hline
    Associative & 7 & $f$ & $f(ij,k) = f(i,jk)$ \\ \hline Leibniz & 7 & $f$ & $f(i,jk) = f(ij,k) + f(j,ik)$ \\ \hline Zinbiel & 7& $f$ & $f(ij,k) = f(i,jk) + f(i,kj)$
    \\ \hline Diassociative & 35 & $(f\dd,f\vv)$ & $f\dd(i,j\dashv k) = f\dd(i,j\vdash k)$ \\ &&& $f\dd(i\vdash j,k) = f\vv(i,j\dashv k)$ \\ &&& $f\vv(i\dashv j,k) = f\vv(i\vdash j,k)$ \\ &&& $f\dd(i,j\dashv k) = f\dd(i\dashv j,k)$ \\ &&& $f\vv(i,j\vdash k) = f\vv(i\vdash j,k)$ \\ \hline Dendriform & 21 & $(f_<,f_>)$ & $f_<(i<j,k) = f_<(i,j<k) + f_<(i,j>k)$ \\ &&& $f_<(i>j,k) = f_>(i,j<k)$ \\ &&& $f_>(i<j,k) + f_>(i>j,k) = f_>(i,j>k)$ \\ \hline Lie & 3 & $f$ & $f(i,j) = -f(j,i)$ \\ &&& $f([i,j],k) + f([j,k],i) + f([k,i],j) = 0$ \\ \hline Commutative & 4 & $f$ & $f(i,j) = f(j,i)$ \\ &&& $f(ij,k) = f(i,jk)$
\end{tabular}

\section*{Acknowledgements}
The author would like to thank Ernest Stitzinger for the many helpful discussions.

\end{document}